\documentclass[11pt]{article}
\usepackage{cite}
\usepackage{lineno}
\usepackage{mathrsfs} 
\usepackage{ifthen}
\usepackage{adjustbox}
\usepackage{siunitx}
\usepackage{multirow}
\usepackage{scalerel}
\usepackage{exscale}
\usepackage{tabularx}
\usepackage{algorithm,algorithmic}     
\usepackage{latexsym}
\usepackage{amsmath,amssymb,amstext,amsthm} 
\usepackage{graphicx,color,epsfig}
\usepackage{graphicx}
\usepackage{caption,subcaption}
\usepackage{fullpage}        
\usepackage{float}
\usepackage{verbatim}
\usepackage{datetime}
\usepackage{makeidx}
\usepackage[inline,shortlabels]{enumitem}
\usepackage{xspace}
\numberwithin{equation}{section}  
\numberwithin{table}{section}
\numberwithin{figure}{section}
\numberwithin{algorithm}{section}
\makeindex

\newlength\myindent
\usepackage[pagebackref=true]{hyperref}
\hypersetup{
plainpages=false,       
unicode=false,          
pdftoolbar=true,        
pdfmenubar=true,        
pdffitwindow=false,     
pdfstartview={FitH},    
pdftitle={Reg. Nonsm. Newt. VII
},    
pdfnewwindow=true,      
colorlinks=true,        
linkcolor=blue,         
citecolor=green,        
filecolor=magenta,      
urlcolor=cyan           
}
\RequirePackage{doi}

\usepackage[noabbrev]{cleveref}


\def\R{\mathbb{R}}

\def\Rn{\R^n}
\def\Rm{\R^m}

\def\Rnp{\R_+^n}
\def\Rnm{\R_-^n}
\def\Rnop{\R_+^{n_1}}
\def\Rnt{\R^{n_2}}

\def\N{\mathbb{N}}

\def\eqref#1{{\normalfont(\ref{#1})}}

\def\LM{\mbox{\boldmath$LM$}\,\,}
\def\LMp{\mbox{\boldmath$LM$}}
\def\RNNM{\mbox{RNNM\,}}
\def\RNNMp{\mbox{RNNM}}
\def\LP{\mbox{\boldmath$LP$}\,\,}
\def\LPp{\mbox{\boldmath$LP$}}
\newcommand{\LPt}{{{\scaleto{\sf LP}{4pt}}}}
\def\HLWB{\mbox{HLWB}\,}
\def\HLWBp{\mbox{HLWB}}
\def\BAP{\mbox{BAP}\,}
\def\BAPp{\mbox{BAP}}
\def\QPPAL{\mbox{QPPAL}\,}
\def\QPPALp{\mbox{QPPAL}}
\def\ADMM{\mbox{ADMM}\,}

\newenvironment{noteH}{\begin{quote}\small\sf HW
                        \color{red} $\clubsuit$~}{\end{quote}}

\newtheorem{theorem}{Theorem}[section]
\newtheorem{defi}[theorem]{Definition}
\newtheorem{definition}[theorem]{Definition}

\newtheorem{corollary}[theorem]{Corollary}

\newtheorem{remark}[theorem]{Remark}

\newtheorem{problem}[theorem]{Problem}

\newtheorem{lemma}[theorem]{Lemma}

\crefname{thm}{Theorem}{Theorems}
\Crefname{thm}{Theorem}{Theorems}
\crefname{problem}{Problem}{Theorems}
\Crefname{problem}{Problem}{Theorems}
\Crefname{assump}{Assumption}{Theorems}
\crefname{assump}{Assumption}{Theorems}
\crefname{conjecture}{Conjecture}{Theorems}
\Crefname{conjecture}{Conjecture}{Theorems}
\crefname{prop}{Proposition}{Propositions}
\Crefname{prop}{Proposition}{Propositions}
\crefname{cor}{Corollary}{Corollaries}
\Crefname{cor}{Corollary}{Corollaries}
\crefname{lem}{Lemma}{Lemmas}
\Crefname{lem}{Lemma}{Lemmas}
\theoremstyle{definition}
\crefname{defn}{definition}{definitions}
\Crefname{defn}{Definition}{Definitions}
\crefname{conj}{Conjecture}{Conjectures}
\Crefname{conj}{Conjecture}{Conjectures}
\crefname{remark}{Remark}{Remarks}
\Crefname{remark}{Remark}{Remarks}
\crefname{rmk}{Remark}{Remarks}
\Crefname{rmk}{Remark}{Remarks}
\crefname{example}{Example}{Examples}
\Crefname{example}{Example}{Examples}
\Crefname{case}{Case}{Cases}
\Crefname{Case}{Case}{Cases}
\crefname{align}{}{}
\Crefname{align}{}{}
\crefname{equation}{}{}
\Crefname{equation}{}{}

\newcommand{\textdef}[1]{\textit{#1}\index{#1}}
\newcommand{\nul}{\mathrm{null}}
\newcommand{\range}{\mathrm{range}}

\newcommand{\cP}{{\mathcal P} }
\newcommand{\cPN}{{\cP_+}}
\newcommand{\TT}{{\mathcal T} }
\newcommand{\cT}{{\mathcal T} }

\newcommand{\cU}{{\mathcal U} }

\newcommand{\norm}[1]{\left\| #1 \right\|}

\newcommand{\cN}{{\mathcal N}}

\newcommand{\A}{{\mathcal A}}

\newcommand{\cI}{{\mathcal I}}
\newcommand{\cB}{{\mathcal B}}
\newcommand{\cZ}{{\mathcal Z}}

\newcommand{\bbm}{\begin{bmatrix}}
\newcommand{\ebm}{\end{bmatrix}}
\newcommand{\bem}{\begin{pmatrix}}
\newcommand{\eem}{\end{pmatrix}}
\newcommand{\beq}{\begin{linenomath*} \begin{equation}}
\newcommand{\beqs}{\begin{linenomath*} \begin{equation*}}
\newcommand{\eeq}{\end{equation} \end{linenomath*}}
\newcommand{\eeqs}{\end{equation*} \end{linenomath*}}
\newcommand{\beqr}{\begin{linenomath*} \begin{eqnarray}}
\newcommand{\beqrs}{\begin{linenomath*} \begin{eqnarray*}}
\newcommand{\eeqr}{\end{eqnarray} \end{linenomath*}}
\newcommand{\eeqrs}{\end{eqnarray*} \end{linenomath*}}
\newcommand{\bet}{\begin{table}}
\DeclareMathOperator{\Proj}{Proj} 

\DeclareMathOperator{\mean}{{mean}}

\DeclareMathOperator{\Diag}{{Diag}}

\DeclareMathOperator{\relint}{{relint}}

\DeclareMathOperator{\rank}{{rank}}

\DeclareMathOperator{\conv}{{conv}}

\newcommand{\nc}{\newcommand}
\nc{\arrow}{{\rm arrow\,}}
\nc{\Arrow}{{\rm Arrow\,}}
\nc{\BoDiag}{{\rm B^0Diag\,}}
\nc{\bodiag}{{\rm b^0diag\,}}

\nc{\Mm}{{\mathcal M}^{m} }
\nc{\Mmn}{{\mathcal M}^{mn} }
\nc{\Mpq}{{\mathcal M}^{pq} }
\nc{\Mnr}{{\mathcal M}_{nr} }
\nc{\Mnmr}{{\mathcal M}_{(n-1)r} }
\nc{\kwqqp}{Q{$^2$}P\,}
\nc{\kwqqps}{Q{$^2$}Ps}
\def\argmin{\mathop{\rm argmin}}

\nc{\notinaho}{(X,S)\in \overline{AHO}(\A)}
\nc{\inaho}{(X,S)\in AHO(\A)}

\newcommand{\bea}{\begin{eqnarray}}%
\newcommand{\eea}{\end{eqnarray}}%
\newcommand{\beas}{\begin{eqnarray*}}%
\newcommand{\eeas}{\end{eqnarray*}}%
\newcommand{\Rmn}{\R^{m \times n}}%
{}

\newcommand{\Hnp}[1][]{\,\mathbb{H}_+^{\ifthenelse{\equal{#1}{}}{n}{#1}}}
\newcommand{\Hn}[1][]{\,\mathbb{H}^{\ifthenelse{\equal{#1}{}}{n}{#1}}}
\newcommand{\Dn}[1][]{\,\mathbb{D}^{\ifthenelse{\equal{#1}{}}{n}{#1}}}

\newcommand{\crb}{\color{blue}}

\begin{document}

\title{
	Regularized Nonsmooth Newton Algorithms\\
	for Best Approximation\\
with Applications
\footnote{{\color{red} PLEASE NOTE} We are including a table of contents, 
lists of tables, index,
to help the referees. We fully intend to delete these before any final 
version of the paper.}
	}

\newcommand*\samethanks[1][\value{footnote}]{\footnotemark[#1]}
\author{
\href{http://math.haifa.ac.il/yair/}
{Yair Censor}\thanks{Departement of Mathematics, University of
Haifa, Mt. Carmel, Haifa 3498838, Israel. Research supported by the
ISF-NSFC joint research plan Grant Number 2874/19 and by U.S. National
Institutes of Health grant R01CA266467.} \and 
\href{https://uwaterloo.ca/combinatorics-and-optimization/about/people/wmoursi}{Walaa M.\
Moursi}\thanks{Department of Combinatorics and Optimization, Faculty of Mathematics, University of Waterloo, Waterloo, Ontario, Canada N2L 3G1; Research supported by The Natural Sciences and Engineering Research Council of Canada} \and 
\href{https://uwaterloo.ca/combinatorics-and-optimization/about/people/tweames}{Tyler Weames}\samethanks  \and
	\href{http://www.math.uwaterloo.ca/~hwolkowi/} {Henry
Wolkowicz}\samethanks%
}

\date{Version 1 Dec. 19, 2022/Revision \today}
          \maketitle



\begin{abstract}
We consider the problem of finding the best approximation point from a
polyhedral set, and its applications, in particular to solving large-scale
linear programs.
The classical {\crb best approximation} problem has many various solution techniques
as well as applications. We study a regularized nonsmooth Newton type
solution method where the Jacobian is singular; and we compare the computational
performance to that of the classical projection method of
Halpern-Lions-Wittmann-Bauschke (\HLWBp).

We observe empirically that the regularized nonsmooth method significantly
outperforms the \HLWB method. However, the \HLWB {\crb method} has a convergence
guarantee while the nonsmooth method is not monotonic
and does not guarantee convergence due in part to singularity of
the generalized Jacobian.

Our application to
solving large-scale linear programs uses a parametrized 
{\crb best approximation} problem. This leads to a {\crb finitely converging} 
\emph{stepping stone external path 
following} algorithm. Other
applications are finding triangles from branch and
bound methods, and generalized constrained linear least squares.
We include scaling methods {\crb and sensitivity analysis
to} improve the efficiency.
\end{abstract}



\tableofcontents
\listoftables
\listofalgorithms
\listoffigures

\section{Introduction}
\label{sect:intro}

The \textdef{best approximation problem, \BAP}, arises in many areas of
optimization and approximation theory. In particular, we study finding
the best approximation $x^*$ to a given point $v$ from a
\textdef{polyhedral set, $P\subset \Rn$}, {\crb in the} $n$-dimensional Euclidean space;
namely, find $x^*(v)\in \Rn$ such that
\index{$P\subset \Rn$, polyhedral set}

\begin{equation} \label{eq:projproblemconv}
x^*(v) = \argmin_{x\in P} \|x-v \|.
\end{equation}
There is an abundance of theory, algorithms, and applications for this
problem, see e.g.,~\cite{1996_Ba,censor-comp-LAA-2006,deutsch-book},
\cite[Chap. 6]{MR3616647}, and the references therein. 
The optimum point $x^*(v)$ is the \emph{projection of $v$
onto the polyhedral set $P$} and is known to be unique. 
In this work we follow a Newton type approach of an 
\emph{elegant} compact optimality
condition, even though the corresponding Jacobian resulting from the 
optimality conditions is
possibly {\crb a generalized Jacobian} and/or singular. 
We include a regularization, as well as an
inexact approach for large-scale problems.
Empirical evidence illustrates the surprising success of this approach.

We include several applications. In particular, we solve
large-scale linear programming, (\LPp), problems using a parametrized
best approximation problem. This introduces an efficient 
{\crb finitely converging,} \emph{stepping stone
external path following} algorithm.
In addition,  we consider large-scale
systems of triangle inequalities.
In our applications we do not assume {\crb differentiability 
of our optimality conditions} and/or
nonsingularity of the generalized Jacobian.
We   introduce a Newton type approach for our
applications that overcomes the nonsmooth difficulties by applying
regularization and scaling. We then provide extensive testing and
comparisons to
\label{page:surprising}
illustrate the surprisingly high efficiency, accuracy, and speed
of our proposed method.

\subsection{{\crb Main Contributions}}
\label{sect:maincontr}
\begin{enumerate}[(i)]
\item First, we present the
basics for the {\crb best approximation} problem, see \Cref{thm:projoptndual} below. This includes an application of
the Moreau decomposition that yields a \emph{single elegant  equation} that 
captures all three KKT optimality conditions: primal and dual
feasibility and complementary slackness.
{\crb This emphasizes the equivalence of this
single equation \cref{eq:xofv}
in the small dimensional dual variable $y$ to solving
the entire KKT optimality conditions. We include a comparison with interior
point methods in~\Cref{rem:compIntPt}.
}

\item Second, we present the nonsmooth, regularized  
Newton method. No line search is used.
(See \Cref{sssection:genralizedjac} below.)
\item
We show that the regularization from a modified, simplified, 
\textdef{Levenberg-Marquardt, \LMp},
method yields a descent direction. (See \Cref{lem:LM} below.)
\index{\LMp, Levenberg-Marquardt}

\item
We present our empirical test results that include an external path
following approach to solving large-scale linear programs that 
fully exploits sparsity. {\crb This is based on efficiently solving the
\BAP
subproblems accurately and applying sensitivity analysis.}
{\crb We} compare our results with several codes in the literature. 
The details are in~\Cref{sect:numerics} below.
\label{page:exploit}
\item
We compare computationally our algorithm with the 
Halpern-Lions-Wittmann-Bauschke, (\HLWBp), algorithm that
belongs to a class of projection methods usually developed and
investigated in the field of fixed point theory.
\end{enumerate}

\subsection{Related Work}
\label{sect:relwork}
Our approach uses a special decomposition from the optimality
\label{page:allows}
conditions that allows for a Newton method with a cone projection 
applied to a system whose
size is of the order of the number of linear equality constraints
forming the polyhedron $P$. This approach first
appeared in infinite dimensional Hilbert space applications,
 \label{page:projL2}
e.g.,~\cite{MR1027508,MR1186970,MiSmSwWa:85,BoWo:86},
where the projection mapping is differentiable, and typically $P$ is the
intersection of a cone and a linear manifold. The approach was
applied to a parametrized quadratic problem to solve finite-dimensional
linear programs in \cite{smw2}. (See our application~\Cref{sect:LPs},
below. In this finite-dimensional case differentiability was lost.)
The approach in infinite-dimensional Hilbert spaces
was followed up and extended in the theory of 
\emph{partially finite programs} in~\cite{BoLe1:92,BoLe2:92} and the
many references therein.
 Further references are given in~\cite{MR0270044,MR1922763,MR1346302}. 

 \label{page:semismooth}
As mentioned above, differentiability is lost in the finite-dimensional
cases, see e.g.,~in~\cite{smw2}. This led to the introduction of
semismoothness~\cite{Mif:77b}. In particular, semismoothness for
a nondifferentiable Newton type method is introduced and applied
in~\cite{QiSun:93nonsmooth,QiSun:06}. Further applications 
for nearest doubly stochastic and nearest Euclidean distance matrices
are presented in~\cite{HuImLiWo:21,homwolkA:04}. A regularized semismooth
approach for general composite convex programs is given
in~\cite{MR3812973}.

Differentiability properties
are nontrivial as discussed in, e.g.,~\cite{MR1539982}. 
A characterization of differentiability in terms of normal
cones is given in~\cite{facchinei2003finite}. Further results and
connections to semismoothness are in, e.g.,~\cite{MR1539982,MR2249554}.
A survey presentation on differentiability properties can be
found at the link~\cite{Sarabi:22}.

\section{Projection onto a Polyhedral Set}
\label{sect:projpolyh}

We begin with the projection onto the polyhedral set  
given in \textdef{standard
form}, since every polyhedron can be transformed into this form.
Suppose we are given $v\in \Rn, b\in \Rm, A\in \Rmn$, $\rank A=m$
{\crb and no columns of $A$ are $0$}. We define the
following \textdef{projection onto a polyhedral set}, i.e.,~the
\textdef{best approximation problem, \BAP} to the \textdef{generalized simplex},

\begin{equation}
\label{eq:projproblemhyp}
    \text{(P)} \qquad  
	  \begin{array}{rcl}
	    x^*(v) := & \argmin_{x}  &\frac12 \norm{x-v}^2 \\ 
	    &  \text{s.t. } &Ax = b  \\
	    &         &x \in \Rnp, \\
         ~~\\
 \text{optimal value:   }    p^*(v) &= & \frac12 \norm{x^*(v)-v}^2,
    \end{array}
\end{equation}
i.e.,~the optimum and optimal value are, respectively, $x^*(v),p^*(v)$;
and $\Rnp$ is the nonnegative orthant.
We now proceed to derive the regularized nonsmooth Newton method,
(\RNNMp) to solve~\cref{eq:projproblemhyp}.
\index{\RNNM, regularized nonsmooth Newton method}
\index{regularized nonsmooth Newton method, \RNNM}
    \index{$p^*(v)$, optimal value}
    \index{optimal value, $p^*(v)$}

\subsection{Basic Theory and Algorithm}
In this section we briefly describe the properties of problem
\cref{eq:projproblemhyp} as well as some background and motivation
behind using a generalized Newton method. 
We assume that

\begin{equation}
	\label{e:def:P}
P:=\{x\in \Rnp : Ax=b\}\neq \varnothing.	
\end{equation}	
Problem \cref{eq:projproblemhyp} has a strongly convex smooth objective 
function and nonempty closed convex constraint set. 
Therefore, the optimal value is finite,
uniquely attained, and strong duality holds. 
In the following, we precisely formulate this conclusion.

Throughout the rest of the paper we set\footnote{Let $x\in \Rn$. 
Here and elsewhere we use $x_+$ (respectively $x_-$) to denote 
the projection of  the vector $x$ onto the nonnnegative orthant 
defined as $x_+=(\max\{0, x_i\})_{i=1}^n$
(respectively  onto the nonpositive orthant defined by 
$x_-=(\min\{0, x_i\})_{i=1}^n$).	
}
\index{squared residual function, $f(y)$}
\index{$f(y)$, squared residual function}

\begin{equation}
	\label{eq:F(y)}
	\textdef{$F(y) := A(v+A^Ty)_+ -b$}, \quad \textdef{$f(y) := \frac 12\|F(y)\|^2$}.
\end{equation}
\begin{theorem}
\label{thm:projoptndual}
Consider the \textdef{generalized simplex
best approximation problem} \cref{eq:projproblemhyp} with 
\textdef{primal optimal value} and optimum $p^*(v)$ and $x^*(v)$,
respectively. Then the following hold:
\begin{enumerate}[(i)]
	\item
	\label{thm:projoptndual:i}
	\index{$p^*(v)$, optimal value}
	The optimum $x^*(v)$ exists and is unique.
	Moreover,
	strong duality holds and
	the \textdef{dual problem} of~\cref{eq:projproblemhyp} is the maximization
	of the \textdef{dual functional, $\phi(y,z)$}:
 \label{pg:projoptndual:ii}

	\[
	p^*(v)=\textdef{$d^*(v)$} := \max_{\substack{z \in \Rnp\\ y\in\Rm }} \phi(y,z) :=  
	-\frac12 \norm{ z+ A^Ty }^2 + y^T(Av - b)-z^Tv.
	\]
	\index{dual functional, $\phi(y,z)$}
	\index{$\phi(y,z)$, dual functional}
\item
\label{thm:projoptndual:ii}
Let $y\in \Rm$. Then

\begin{equation}
	\label{eq:xofv}
F(y)=0 \, \iff \,   y \in \argmin_u f(u) \text{  and  } 
 x^*(v) =  (v+A^Ty)_+.
\end{equation}
\end{enumerate}
\end{theorem}

\begin{proof}

Recall that the \textdef{Lagrangian $L(x,y,z)$}  for~\cref{eq:projproblemhyp}, and
its gradient, are respectively
\index{$L(x,y,z)$, Lagrangian}

\begin{equation} 
\label{proj prob aff}
 	 L(x,y,z) = \frac12 \norm{x-v}^2 + y^T(b-Ax)-z^Tx, \quad
 	 \nabla_x L(x,y,z) = x-v - A^Ty -z.
\end{equation}

\ref{thm:projoptndual:i}:
The solution of the
problem \cref{eq:projproblemhyp} is a projection onto a nonempty polyhedral set,
which is a closed and convex set, see \cref{e:def:P}.
 Therefore, the optimum exists and is unique
and strong duality holds, i.e.,~there is a zero duality gap and the dual
is attained.

Let $x$ be a stationary point of the Lagrangian i.e.,~$\nabla_x L(x,y,z)=0$. 
Then by~\cref{proj prob aff} we have the following 
equivalent representation

\[
 x = v + A^Ty +z.
\]
It then follows that at a stationary point $x$ we have

\[
\begin{array}{rcl}
 L(x,y,z) 
&=&
 \frac12 \norm{v + A^Ty +z-v}^2 + y^T(b-A(v + A^Ty +z) )-z^T(v + A^Ty +z)
\\&=&
 \frac12 \norm{ A^Ty +z}^2 + y^Tb -y^TAv -(A^Ty)^T(A^Ty+z) 
-z^Tv - z^T(A^Ty +z)
\\&=&
 \frac12 \norm{ A^Ty +z}^2 + y^Tb -y^TAv -(A^Ty+z)^T(A^Ty+z) -z^Tv
\\&=&
 -\frac12 \norm{ z+ A^Ty }^2 + y^T(b-Av)-z^Tv.
\end{array}
\]

The Lagrangian dual is 

\[
\begin{array}{rcll}
d^*
&=&
    \max_{y\in\Rm\!,z \in \Rnp } \min_{x\in \Rnp} & L(x,y,z) \quad (= \frac12 \norm{x - v}^2 +
    y^T(b-Ax ) - z^Tx)
\\&=&
    \max_{x\in \Rnp\!,y\in\Rm,z \in \Rnp } &  \{L(x,y,z) : \nabla_x L(x,y,z) = 0\}
\\&=&
    \max_{x\in \Rnp\!,y\in\Rm,z \in \Rnp} &  \{L(x,y,z) : x = v + A^Ty +z\}
\\&=&
    \max_{y\in\Rm\!,z \in \Rnp} &  -\frac12 \norm{ z+ A^Ty }^2 + y^T(b-Av)-z^Tv.
\end{array}
\]
Moreover, $p^*:=p^*(v)=d^*:=d^*(v)$, and the dual value is attained. 

\ref{thm:projoptndual:ii}:
Now the \textdef{KKT optimality conditions} 
for the primal-dual variables $(x,y,z)$ 
\label{pg:nonnegpolar}
are\footnote{Let $S\subset \Rn$.
We use   $S^+ = {\crb \{\phi : \langle \phi,s\rangle \geq 0, \forall
s\in S\}}$ to denote the {\crb (nonnegative)} polar cone of the set $S$.}:

\begin{equation*}
    \begin{array}{ll}
{\nabla_x} L(x,y,z) = x - v - A^Ty - z  = 0,\, z \in \R^n_+, & \text{(dual feasibility)} \\ 
{\nabla_y} L(x,y,z) = Ax - b  = 0,\, x \in \R^n_+, & \text{(primal feasibility)}\\ 
{\nabla_z} L(x,y,z) \cong  x  \in (\Rnp-z)^+. &
\text{(complementary slackness $z^Tx=0$)}
    \end{array}
\end{equation*}
\index{polar cone, $S^+$}
\index{$S^+$, polar cone}
The above KKT conditions can be rewritten as : 

\begin{equation} 
	\label{eq:KKTprojproblemhyp}
		\begin{pmatrix} x - v - A^Ty - z \\ Ax - b \\ z^Tx \end{pmatrix}
	= \begin{pmatrix} 0 \\ 0 \\ 0\end{pmatrix}, \quad x,z \in \R^n_+, y \in
	\R^m.
\end{equation}
It follows from the dual feasibility that 
$v+A^Ty = x-z = x+(-z)$. Together with the complementary slackness we have  

\[
x^Tz = 0,\, x,z \in \Rnp,\, -z \in \Rnm = (\Rnp)^+,
\]
and we learn that $x - z$  is the  Moreau decomposition of $v+A^Ty$. 
That is

\begin{equation}
\text{$x = (v+A^Ty)_+$ and  $-z=(v+A^Ty)_-$; equivalently, $  z=-(v+A^Ty)_-$. }
\end{equation}	
Substituting for $x = (v+A^Ty)_+$ we obtain a
simplification of the optimality conditions in \cref{eq:KKTprojproblemhyp}  
as follows

\[
A(v+A^Ty)_+ = b,\, x= (v+A^Ty)_+
\implies z=-(v+A^Ty)_-, z^Tx=0,\, x,z\in \Rnp,
	\,x - v - A^Ty - z =0,
\]
equivalently; $F(y)=0$, for some $y\in \Rm$.

{\crb 
 \label{page:proofconvopt}
For the converse, let $y\in \Rm$ be given and
suppose that $F(y)=0$. Let $\bar x = (v+A^Ty)_+$. Therefore, $\bar x$ is
primal feasible. Let $\bar z= -(v+A^Ty)_-$. We get 
nonnegative feasibility and complementary slackness: $\bar z\geq 0, \,\bar
z^T\bar x=0$. And,
\[
(v+A^Ty)=\bar x - \bar z \implies \bar x -v - A^Ty -z = 0,
\]
i.e.,~dual feasibility holds. The KKT conditions now imply
that  $\bar x(v)$ is optimal. Moreover, $F(y)=0$ implies that $y\in
\argmin_u f(u)$, i.e.,~$y$ solves the nonlinear least squares problem.}
\end{proof}
{\crb
\begin{remark}
\label{rem:compIntPt}
Interior point methods use perturbed KKT conditions with $z^Tx=0$ in
\cref{eq:KKTprojproblemhyp} replaced by 
$z_jx_j = \mu, x_j>0,z_j>0, \forall j$, where $\mu>0$ is the log-barrier
parameter. A Newton step is taken with backtracking to stay strictly
feasible.
Therefore, our method is equivalent to fixing $\mu=0$
throughout the iterations and not staying strictly feasible for $x,z$. 
This is comparable to the predictor step in
predictor-corrector methods, or to affine scaling method.
\end{remark}
}

\subsubsection{Nonlinear Least Squares;  Jacobians}
\label{sssection:genralizedjac}
The \BAP as described in \cref{eq:projproblemhyp} 
is equivalent to the minimization of $f(y)$ in~\cref{eq:F(y)}, 
i.e,~to a nonlinear least squares problem where the
nonlinearity arises from the projection.

This system can be recharacterized by introducing the (possibly nonsmooth)
projection of a vector $p$ onto the nonnegative,  respectively nonpositive,
orthant denoted \textdef{$p_+  = \argmin_x \{\|x-p\| : x \geq 0\}$},
respectively \textdef{$p_-  = \argmin_x \{\|x-p\| : x \leq 0\}$}.
In general, we can define the \textdef{Moreau decomposition} of $p$ with
respect to $\Rnp$ as $p = p_+ + p_-, \, p_+^T p_- = 0$.

Note that in the differentiable case the gradient of the squared
residual $f(y)$ in~\cref{eq:F(y)} is
 \label{page:NewtonFFs}

\[
\nabla f(y) = (F^\prime(y))^*F(y),
\]
where $(\cdot)^*$ denotes the adjoint (here adjoint is
transpose) and $F^\prime$ denotes the Jacobian matrix.
We note that we have differentiability of the function
$h(w):=w_+$ if, and only if, $\{i : w_i=0\}=\varnothing$ if, and only
if, $w-w_+$ is in the relative interior of the normal cone of $\Rn_+$
at $w_+$ (negative of the polar cone at $w_+$), 
see~\cite[Page 7]{Sarabi:22},\cite{facchinei2003finite}.

We now discuss the framework of nonsmooth terminology needed for
generalized gradients of a general function $H  : \Rn\to \Rn$.

\label{page:HF}

\begin{defi}[(local) Lipschitz continuity]
Let $\Omega \subseteq \Rn$.
A function $H:\Omega \to \Rn$ is \textdef{Lipschitz continuous} on $\Omega$
if there exists $K > 0$ such that 

\[
\|H(y) -H(z)\| \leq K \|y - z\|, \, \forall y,z \in \Omega.
\] 
$H$ is \textdef{locally Lipschitz continuous} on $\Omega$ if for each
$x\in \Omega$ there exists a neighbourhood $U$ of $x$ such that $H$ is
Lipschitz continuous on $U$.
\end{defi}

Let $\Omega \subseteq \Rn$.
It follows from Rademacher's Theorem \cite{Rademacher,MR41:1976} that
if $H:\Omega \to \Rn$ is locally Lipschitz on $\Omega$ then $H$
is Frech\'et differentiable almost everywhere on $\Omega$.
Following Clarke \cite[Def. 2.6.1]{Clarke:83}, we recall the 
following definition of the
\textdef{generalized Jacobian}\footnote{For our application  we restrict
ourselves to square Jacobians.}. 

\begin{defi}[generalized Jacobian] 
Suppose that $H:\Rm \to \Rm$ is locally Lipschitz.

\label{page:Fdiff}
Let $D_H$ be the set of points where $H$ is differentiable.
Let $H^\prime(y)$ be the usual Jacobian matrix at $y\in D_H$. The
\textdef{generalized Jacobian of $H$ at $y$, $\partial H(y)$}, is
the convex hull\footnote{Let $S\subset \Rn$. The convex hull of $S$,
denoted $\conv(S)$ is the smallest convex set containing $S$.} of
the set of all matrices obtained as limits
of usual Jacobians, defined as follows

\[
\partial H(y) := \conv \left\{\lim\limits_{\stackrel{y_i \to y}{y_i\in D_H}}
H^\prime(y_i) \right\}.
\]
In addition, $\partial H(y)$ is called nonsingular if every $V\in \partial
H(y)$ is nonsingular.
\end{defi}


{\crb 
We now return to the nonlinear least squares problem~\cref{eq:F(y)} with
functions $f$ and $F$.}
In the differentiable case, 
the {\crb Gauss-}Newton direction is the solution of the {\crb (consistent)}
{\crb Gauss-}Newton equation\footnote{\crb The Gauss-Newton direction is the
minimum of the quadratic model $f(y+\Delta y) \approx f(y) + \nabla
f(y)^T\Delta y + \frac 12 \Delta y^T((F^\prime(y))^*F^\prime(y))\Delta y$,
i.e.,~the higher order quadratic terms are ignored, e.g.,~\cite{gil81}.
This is particularly suitable here as the higher order terms involve the
$F(y)$ that is converging to zero.}
 \label{page:NewtonHHs}

\begin{equation}
\label{eq:Newtdir}
F^\prime(y))^*(F^\prime(y)) \Delta y = -(F^\prime(y))^*F(y).
\text{  (equivalently {\crb invertible case}, }
  F^\prime(y) \Delta y = -F(y)).
\end{equation}
In the sequel $A^\dagger$ denotes
the generalized (Moore-Penrose) inverse of a matrix $A$.
Solving for {\crb the best least squares solution}
$\Delta y$ in~\cref{eq:Newtdir} yields
\index{$A^\dagger$, generalized inverse}
\index{generalized inverse, $A^\dagger$}

\begin{equation}
\label{eq:Newtdirdy}
{\crb \Delta y =  -F^\prime(y))^\dagger F(y).
}
\end{equation}
Therefore, the directional derivative of $f$ in the direction 
$\Delta y$ satisfies

{\crb 
\begin{equation}
\label{eq:descentdir}
\begin{array}{rcl}
\Delta y^T \nabla f(y) 
&=& (F^\prime(y))^\dagger F(y))^T
 (- (F^\prime(y))^*F(y))
\\&=& -\|\Proj_{\range((F^\prime(y))^*)}F(y))\|^2
\\&<& 0, \quad \text{if  } F(y) \notin \nul((F^\prime(y))^*),    
\end{array}
\end{equation}
where $\Proj_{\Omega}(u)$ denotes the orthogonal projection of the point 
$u$ onto the set $\Omega$.
We conclude in the differentiable case that: the Gauss-Newton direction 
$\Delta y$ is a descent direction when $F(y)\neq 0$.}

The \textdef{Levenberg-Marquardt, \LMp}, method is a popular method
for handling singularity in 
$F^\prime(y)$ by using the substitution/regularization
$(F^\prime(y))^*F^\prime(y) \leftarrow
\left((F^\prime(y))^*F^\prime(y)\right) +\lambda I,\, \lambda >0$.
{\crb We now see that we maintain a descent direction with a similar
simplified approach if the basic assumption in \cref{eq:notzero} holds.
This simplified approach avoids the product
$(F^\prime(y))^*F^\prime(y)$ and thus avoids increased ill-conditioning
and loss of sparsity.
}
\begin{lemma}
	\label{lem:LM}
{\crb Consider the nonlinear least squares problem in~\cref{eq:F(y)}.
Let $y\in \Rm$, with $F$ differentiable at $y$.
} 
Let $\lambda > 0$ and let $\Delta y$ be the (unique) solution of 

\[
(F^\prime(y) +\lambda I) \Delta y = -F(y).
\]
Then $F^\prime(y)$ is positive semidefinite, $F^\prime(y)\succeq 0$,
and moreover, $\Delta y$ is the {\crb simplified \LM direction}
and is a descent direction if, and only if,
\begin{equation}
\label{eq:notzero}
F(y) \neq 0.
\end{equation}
\end{lemma}
\begin{proof}
For simplicity, set $J=J(y)=F^\prime(y)$.
{\crb By the feasibility assumption for~\cref{eq:projproblemconv}, we
conclude that $0 = \min_y f(y)$ and that the basic assumption satisfies
\begin{equation}
\label{eq:FnJF}
F(y) \neq 0 \iff JF(y) \neq 0.
\end{equation}
We observe that $J$ is {\crb symmetric positive semidefinite}
follows from the definitions; see~\cref{eq:maxrankM} below.
}
Let $J=U D U^T$ denote the orthogonal spectral decomposition.
The simplified regularization of \LM type uses
$(J+\lambda I)\Delta y =  -F$. Therefore,
{\crb
 \label{page:LMproof}

\[
\Delta y = -\left(J+\lambda I\right)^{-1} F
= -U\left(D+\lambda I\right)^{-1} U^TF.
\]
Therefore, the directional derivative of $f$
at $y$ in the direction of $\Delta y $ is

\[
\begin{array}{rcl}
\Delta y^T \nabla f(y) 
&=&
 -\left(U\left(D+\lambda I\right)^{-1} U^TF\right)^T (UDU^TF)
\\&=&
 -(U^TF)^T\left(D+\lambda I\right)^{-1} D(U^TF)
\\&=&
 -(U^TF)^TD^{1/2}\left(D+\lambda I\right)^{-1} D^{1/2}(U^TF)
\\&<&
 0 \quad \iff   (D^{1/2}U^T)F \neq 0.
\end{array}
\]
}
By \cref{eq:FnJF}, the latter is not zero if, and only 
if,~\cref{eq:notzero} holds. This completes the proof.

\end{proof}

\subsubsection{{\crb Well Conditioned} Generalized Jacobian}
Recall the optimality conditions derived following 
\cref{eq:KKTprojproblemhyp}.
If we denote the orthogonal projection operator onto the nonnegative orthant by
\textdef{$\cPN w  = w_+$}, then

\[
Aw_+ = A (\cPN w)  = (A\cPN)w_+ = (A\cPN)(\cPN w) = \sum_{w_i> 0} w_iA_i.
\]
Here $A_i$ is the $i$-th column of $A$.
Thus, we see that at points where the projection is differentiable,
the columns of $A$ that are chosen correspond to the
positive variables of $w$. We note that 
 \label{page:Ayeq0}

\[
{\crb v+A^Ty> 0} \implies  F^\prime(\Delta y) = AIA^T\Delta y= AA^T\Delta y.
\]
Define the three index sets, {\crb $\cI_{+},\, \cI_{0},\, \cI_{-}$}, respectively, by
 \label{page:signfunfix}

{\crb 
\[
\textdef{$\cI_{+,0,-} := \cI_{+,0,-}(y) = \{i : (v+A^Ty)_i >0,=0,<0\}$}.
\]
Then, for sufficiently small $\Delta y$ we can ignore $\cI_-$ to get
\[
\begin{array}{rcl}
F(y+\Delta y) - F(y) 
&=&
A(v+A^T(y+\Delta y))_+ - A(v+A^Ty)_+
\\&=&
\sum_{i\in \cI_+(y+\Delta y)} (v+A^T(y+\Delta y))_iA_i 
 -\sum_{i\in \cI_+(y)} (v+A^Ty)_iA_i  
\\&=&
\sum_{i\in \cI_+(y)} (A^T\Delta y)_i A_i +
\sum_{i\in \cI_+(y+\Delta y)\cap \cI_0(y)} (v+A^T(y+\Delta y))_i  A_i
\\&=&
\sum_{i\in \cI_+(y)} A_iA_i^T\Delta y +
\sum_{i\in \cI_+(y+\Delta y)\cap \cI_0(y)} (v+A^T(y+\Delta y))_i  A_i
\\&=&
\sum_{i\in \cI_+(y)} A_iA_i^T\Delta y +
\sum_{i\in \cI_+(y+\Delta y)\cap \cI_0(y)} (A^T\Delta y)_i  A_i
\\&=&
\sum_{i\in \cI_+(y)} A_iA_i^T\Delta y +
\sum_{i\in \cI_+(y+\Delta y)\cap \cI_0(y)} A_iA_i^T\Delta y.
\end{array}
\]
We note that the first summation is over the fixed index set $\cI_+(y)$,
while the second is dependent on $(A^T \Delta y)_i>0$.
Suppose that $A_{\cI_0}^T \Delta y = e_i$ is consistent for each $i\in\cI_0$. Then we can add or not add the corresponding column to the
generalized Jacobian. This means we only need a maximum linearly
independent subset of the columns $A_{\cI_0}$. Let
\textdef{$\bar \cI_0 \subseteq \cI_0$} be a  maximum linearly independent
subset\footnote{We use the variant of the QR decomposition
\href{https://www.mathworks.com/matlabcentral/fileexchange/77437-extract-linearly-independent-subset-of-matrix-columns}{\emph{licols}}
to extract a \emph{nice} subset of linearly independent columns.}.
}

Following~\cite{HuImLiWo:21} {\crb with the change using \emph{licols}}
and $\bar \cI_0$, we define the following set 

\begin{equation}
\label{Uy}
\textdef{$\cU(y)$}:= \left\{u \in \Rn \,:\, 
u_i \in 
\begin{array}{cl}
&\\
\left\{\begin{array}{cl}
\{1\}, & \text{ if } i \in \cI_+ \\
 \left[0,1\right],  & \text{ if } i \in \bar \cI_0 \\
\{0\}, & \text{ if } i \in \cI_-\cup (\cI_0\backslash \bar \cI_0)
\end{array} \right. \\
&
\end{array} 
\right \}.
\end{equation}
Then the generalized Jacobian of the nonlinear system at $y\in \Rm$ is given 
by the set 

\begin{equation}
\label{eq:genjacU}
    \partial F(y) = \{A \, \Diag(u) \, A^T\,:\, u \in \cU (y) \}.
\end{equation}
Let $y_0\in \Rm$.
Here $\Diag$ is the diagonal matrix formed from $u$.
The nonsmooth Newton method for solving $F(y) = 0$ consists of 
the following iterative process.
\index{$\Diag(v)$}

\begin{equation} \label{eq:iter}
	y^{k+1} = y^k - V_k^{-1} F(y^k), \, V_k \in \partial F(y^k).
\end{equation}
{\crb Here $V_k$ is a generalized Jacobian (matrix) taken from the generalized
Jacobian $\partial F(y^k)$.}

 \label{page:signfun}
We note that, defining $M=\Diag(u)$ with $\crb{ u\in \cU(y)}$, we have
 \label{page:defu}

{\crb
\begin{equation}
\label{eq:maxrankM}
    AMA^T 
= \sum_{i\in\cI_+\cup \bar \cI_0}u_i A_iA_i^T,\quad
u_i=1, i\in \cI_+,\,
u_i\in [0,1], i\in \bar \cI_0.\footnote{\crb Note that for positive diagonal $M$,
and rectangular $B$, the ranks of $B, BM, (BM)(BM)^T$ are all the same.}
\end{equation}
Note that for an index set $\cT$, $A_\cT$ denotes the submatrix 
of $A$ formed using the columns indexed by $\cT$.
\index{$A_\cT$, columns of $A$} 
}

{\crb
\begin{remark}
\label{rem:genJac}
Since we have freedom in choosing the 
values $u_i\in [0,1], i\in \bar \cI_0$, we follow the optimal diagonal
scaling in \cite[Prop. 2.1(v)]{DeWo:90}, \cite[Thm. 5.2]{HaLeWaDaHeCondNumb:23}
to minimize a condition number, and choose the generalized Jacobian by
setting
\[
u_i = \min\{1, 1/\|A_i\|^2\}, \, \forall i \in \bar \cI_0.
\]
This means that the generalized Jacobian matrix we choose is nonsingular if, 
and only if, $A_{\cI_+\cup\cI_0}$ is full rank $m$. Moreover, for large
problems we expect $\|A_i\|>1$ and therefore $u_i<1$. This goes against
the intuitive choice of making $u_i$ as large as possible, i.e.,~$=1$.
Note that all elements of $\partial F(y)$ are invertible if, and only
if, $A_{\cI_+}$ is invertible; while there
exists an invertible element if, and only
if, $A_{\cI_+\cup \cI_0}$ is full rank $m$.
\end{remark}
}
 \label{page:Acols}

\subsubsection{Vertices and Polar Cones}
In our numerical tests we can decide on the characteristics of the optimal
solution using the properties of (degenerate) vertices.

\begin{lemma}[\textdef{vertex} and \textdef{polar cone}]
\label{lem:xoptvertex}
Suppose that \textdef{$x(y) = (v+A^Ty)_+ \in P$},  where $y\in \Rm$. 
Then the following are equivalent:
\begin{enumerate}[(i)]
\item
\label{i:isvertex}
$x(y)$ is a vertex of $P$;
\item
 \label{page:Acolsnonsing}
$A_{\cI_+(y)}$ is full column rank;
\item
\label{i:ismatrix}
$\begin{bmatrix}
A_{\cI_+}&A_{\cI_0\cup \cI_-}\cr
 0     &     I_{\cI_0\cup \cI_-}
\end{bmatrix}$
is full column rank $n$.

\end{enumerate}
Moreover: 
\begin{enumerate}[(a)]
\item
\label{item:nonsingJac}
the corresponding generalized Jacobian in \cref{eq:maxrankM},
\Cref{rem:genJac}, is nonsingular if $x(y)$ is a nondegenerate vertex;
\item
\label{item:nonnegpolar}
the {\crb (nonnegative)} polar cone of the feasible 
set $P$ at $x = x(y)$ is
\index{feasible set, $P$}
\index{$P$, feasible set}

\begin{equation}
\label{eq:polarconevrtx}
(P-x)^+ = \{w : w = A^Tu + z,\, u \in  \Rm,\, z\in \Rnp,\, x^Tz=0\}.
\end{equation}
\index{polar cone of $P$ at $x$, $(P-x)^+$}
\index{$(P-x)^+$, polar cone of $P$ at $x$}
\end{enumerate}
\end{lemma}
\begin{proof}
Without loss of generality we can permute the columns of $A$
and corresponding components of $x$
and have $A=\begin{bmatrix} A_{\cI_+}&A_{\cI_0}&A_{\cI_-}\end{bmatrix}$. 
We know that $x(y)$ is a vertex (equivalently an extreme point, a 
basic feasible solution) if, and only if $A_{\cI_+}$ can be completed to
a basis matrix if, and only if, the active set is full rank $n$.
The active set of constraints is

\begin{equation}
\label{eq:actset}
\begin{bmatrix}
A_{\cI_+}&A_{\cI_0\cup \cI_-}\cr
 0     &     I_{\cI_0\cup \cI_-}
\end{bmatrix} x = 
\begin{pmatrix} b \cr 0
\end{pmatrix}.
\end{equation}
This has the unique solution $x(y)$ if, and only if, $A_{\cI_+}$ is
full column rank. {\crb This shows the three equivalences 
\cref{i:isvertex,page:Acolsnonsing,i:ismatrix}, as well as
the nonsingularity of the generalized Jacobian that we choose
as claimed in~\cref{item:nonsingJac}.}

From the optimality conditions we have that the gradient of the
objective satisfies 

\[
x-v = A^Ty+ \sum_{j\in \cI_0\cup \cI_-} z_j e_j,
\]
where $e_j$ is the $j$-th unit vector.
And we know that $x-v$ is in the polar cone at $x$ if, and only if, $x$
is optimal. Therefore, this yields the description of the polar cone at $x$
as claimed in~\cref{item:nonnegpolar}.
\end{proof}

\begin{remark}[degeneracy of optimal solutions]
\label{rem:degoptsols}
Let $x$ be a boundary point of $P$.
Then the polar cone of $P$ at $x$ is given in~\cref{eq:polarconevrtx}.
Moreover,  $x$ is the optimal solution of~\cref{eq:projproblemhyp}
if, and only if, $x-v \in (P-x)^+$, i.e.,~we can choose $v$ with

\[
v = x - A^Tu {\crb -z}, \, z\geq 0, \, z^Tx=0. 
\]
In fact, we can choose $z$ so that $x+z>0$ and have no degeneracy or
choose $z=0$ and have high degeneracy. For these choices we still get
$x$ optimal.
As mentioned above, it is shown in~\cite{facchinei2003finite} that

\[
x^*(v) \text{  is differentiable at $\bar v$  } \iff \,\,
(x^*(\bar v)-\bar v)  \in \relint (P-x^*(\bar v))^+,
\]
where $\relint$ refers to the relative interior.
This justifies our use of the Levenberg-Marquardt regularization.
\end{remark}

The pseudocodes for solving
\Cref{eq:projproblemhyp} using the exact and inexact nonsmooth Newton
methods are presented below  in~\Cref{app:pseudos} in 
\Cref{alg:ExactNonsmoothNewton,alg:InexactNonsmoothNewton},
respectively.

\section{Cyclic \HLWB Projection for Best Approximation}
\label{sec:BAP}

\index{\HLWB, Halpern-Lions-Wittmann-Bauschke}
\index{Halpern-Lions-Wittmann-Bauschke, \HLWB}
A notable aspect of this work is the computational comparison
of our semismooth algorithm with the method 
of Halpern-Lions-Wittmann-Bauschke, (\HLWBp). The convergence analysis of
the method has its roots in the field of fixed point theory. For
the readers' convenience we provide a brief description and some
relevant references.

\begin{problem}
[{{\crb The} \textdef{best approximation problem for linear inequalities}}]
\label{prob:BAP}
Given an $m\times n$ matrix $A$ and a vector $b\in R^{m}$ such that

\begin{equation}
Q:=\{x\in R^{n} : Ax\leq b\}\neq\varnothing,\label{eq:Q}
\end{equation}
and a point $v\in R^{n},$ $v\notin Q,$ called the \textdef{anchor point}, find
the orthogonal projection of $v$ onto $Q,$ denoted by $P_{Q}(v).$
\end{problem}


The set $Q$ is the intersection of $m$ half-spaces.
Denote the $i$-th half-space of \cref{eq:Q} by

\begin{equation}
\label{eq:Hihalfspace}
H_{i}:=\{x\in R^{n} \,:\, x^Ta^i\leq b_{i}\},
\end{equation}
where $a^i$ is the $i$-th row of $A$ and $b_i$ is the $i$-th component
of $b$.
The orthogonal projection of a point
$v\in R^{n}$ onto $H_{i}$, denoted by $P_{i}(v)$, is
\index{$a^i$, $i$-th row of $A$}

\begin{equation}
P_{i}(v)=v+\min\left\{ 0,\frac{b_{i}-x^Ta^{i}}{\left\Vert a^{i}\right\Vert ^{2}}\right\} a^{i}.\label{eq:proj}
\end{equation}
The \HLWB algorithm for this problem is a \textit{projection method}
that employs projections onto the individual half-spaces of 
\cref{eq:Hihalfspace}
and makes use of a sequence of, so called, steering parameters.
\begin{definition} [\textdef{steering sequence}]
\label{def:Steering-params}
A real sequence
$(\sigma_{k})_{k=0}^{\infty}$ is called a \textit{steering sequence}
if it has the following properties:

\begin{equation}
\label{eq:SteeringSeq}
\begin{array}{ll}
\sigma_{k}\in[0,1]\enspace\text{for}\enspace\text{all}\enspace\,k\geq0,\enspace\text{and}\enspace\underset{k\rightarrow\infty}{\mathrm{lim}}\sigma_{k}=0,&
\\
\vspace{.06in}
\sum_{k=0}^{\infty}\sigma_{k}=\infty,  
&  \left(\text{\text{or equivalently, }} 
           \prod_{k=0}^{\infty}(1-\sigma_{k})=0\right),
\\
\sum_{k=0}^{\infty}|\sigma_{k+1}-\sigma_{k}|<\infty.&
\end{array}
\end{equation}
\end{definition}
Observe that although $\sigma_{k}\in[0,1]$, the definition rules out
the option of choosing all $\sigma_{k}$ equal to zero or all equal
to one because of contradictions with the other properties. 
The third property in \cref{eq:SteeringSeq}
was introduced by Wittmann, see, e.g., the review paper of López,
Martin-Márquez and Xu \cite{lopez-2010}.\smallskip{}

\begin{algorithm}
\caption{cyclic \HLWB algorithm for linear inequalities}
\label{alg:HLWB}
\textbf{Initialization:} Choose an arbitrary initialization
point {\crb $x_{0}$} $\in R^{n}$\; 
\\\textbf{Iterative Step:}~Given the current
iterate {\crb $x_{k}$}, calculate the next iterate {\crb $x_{k+1}$} by 

\begin{equation}
{\crb x_{k+1}}=\sigma_{k}v+(1-\sigma_{k})P_{i_k}({\crb x_{k}}),  \label{eq:basic}
\end{equation}
where $v$ is the given anchor point, $i_k=k$ mod $m$$+1$ and
$(\sigma_{k})_{k=0}^{\infty}$ is a steering sequence.
\end{algorithm}

\smallskip{}

\index{best approximation problem, \BAP}
\index{BAP, best approximation problem}
The \HLWB algorithm has a much broader formulation that applies to
the  \BAP with respect to the common fixed
points set of a family of firmly nonexpansive (FNE) operators presented
in Bauschke \cite{1996_Ba};  see also Bauschke and Combettes
\cite[Chap. 30]{MR3616647}. For more on the \BAP,
see, e.g., Deutsch's book \cite{deutsch-book}. The family of iterative
projection methods for the \BAP includes, in addition to the \HLWB method,
also Dykstra's algorithm \cite{BD86}, \cite[Theorem 30.7]{MR3616647},
Haugazeau's algorithm \cite{Haugazeau}, \cite[Corollary 30.15]{MR3616647},
and Hildreth's algorithm \cite{Hildreth,LC-hildreth1980}. There are
also simultaneous versions of some of these algorithms available,
see, e.g., \cite{censor-comp-LAA-2006}. A string-averaging \HLWB algorithm,
which encompasses the sequential, the simultaneous and other variants
of the \HLWB algorithm, recently appeared in \cite{nisenbaum}.

More on applications of \BAP and the \HLWB algorithm are given
in~\Cref{subsec:Applications}.

\section{Applications}
We consider several applications of the best approximation
problem,~\cref{eq:projproblemhyp}.
Of special interest is the following approach
to solving a linear program, (\LPp).
\index{linear program, \LPp}
\index{\LPp, linear program}

\subsection{Solving Linear Programs}
\label{sect:LPs}
We consider a maximization primal LP in standard equality form 

\begin{equation}
\label{eq:PLP}
    \text{(PLP)} \qquad  
	  \begin{array}{rcl}
	    p_{LP}^* := & \max  & c^Tx \\
	    &  \text{s.t. } &Ax = b\in \Rm  \\
	    &         &x \in \Rnp.
    \end{array}
\end{equation}
The dual LP is

\begin{equation}
\label{eq:DLP}
    \text{(DLP)} \qquad  
	  \begin{array}{rcl}
	    d_{LP}^* := & \min  & b^Ty \\
	    &  \text{s.t. } &A^Ty -z = c\in \Rn \\
	    &         &z \in \Rnp.
    \end{array}
\end{equation}
We assume that $A$ is full row rank and that
the optimal value is finite. Note that the fundamental
theorem of linear programming now guarantees that strong duality holds for
both the primal and dual problems, i.e.,~equality $p_{LP}^*=d_{LP}^*$
holds and both optimal values are \emph{attained}. 

We now see in \Cref{lem:LPxR} that the solution to (PLP) is
the limit of the sequence of projections of the
vectors $v_R=Rc\in \Rn$ onto the feasible set as\footnote{Note that our
algorithm identifies infeasibility, but we do not consider that
aspect in this paper.}
$R\uparrow \infty$.
\begin{lemma}[\cite{MR83c:90098,MR774243,MR2062967,smw2}]
\label{lem:LPxR}
Let the given LP data be
$A,b,c$ with finite optimal value $p_{LP}^*$. For each $R>0$ define

\begin{equation}
\label{eq:LPxR}
	  \begin{array}{rcl}
	    x^*(R) := & \argmin_{x}  &\frac12 \norm{x-Rc}^2 \\ 
	    &  \text{s.t. } &Ax = b\in \Rm  \\
	    &         &x \in \Rnp.
    \end{array}
\end{equation}
Then $x^*$ is the \textdef{minimum norm solution} of (PLP) if, and only if,
there exists $\bar R > 0$ such that

\begin{equation}
R \geq \bar R \implies
x^* = x^*(R)=  \argmin \left\{ \frac12 \norm{x-Rc}^2 \,:\,
	    Ax = b, \, x \in \Rnp \right\}.
\end{equation}
\end{lemma}
{\crb
\begin{remark}
Note that the objective function in~\cref{eq:LPxR} when expanded is equivalent to
$R(-c^T x + \frac 1{2R} \norm{x}^2) + (\frac 12\|Rc\|^2)$, i.e.,~this is
equivalent to minimizing $-c^T x + \frac 1{2R} \norm{x}^2$,
an \emph{exact} regularization of the original \LP~\cref{eq:PLP},
e.g.,~\cite{SaundersTomlin:96,FrieTsen:2007}. In fact, using a Lagrange
multiplier argument, we observe that this is equivalent to adding a 
trust region constraint $\|x\|^2 \leq \delta$ to the \LPp. 
The trust region radius
$\delta$ is inversely proportional to the regularization parameter
$\frac 1{2R}$ and so directly proportional to $R$, for $R\leq \bar R$,
where $\bar R$ is given in \Cref{lem:LPxR}.
We note that if $\delta$ is too small, we would have an infeasible
problem. Equivalently, if $R$ is too small, then the BAP solution
$x^*(R)$ is not near the optimal solution $x^*$ of the \LPp.

In our application, we ignore the regularization property but exploit
the fact that we can solve the \BAP efficiently for each $R$.
\end{remark}
}

We would like an $R$ that is not too large but
large enough so that $Rc > \|x^*\|$. We use the following estimate to
start our algorithm:

\begin{equation}
\label{eq:Restimate}
R = \min\left\{50,\frac {\sqrt{mn}\norm b }{1+\norm c }\right\}.
\end{equation}
To avoid numerical complications from large numbers, we consider the
following equivalent problem that uses the scaling $\frac 1Rb$ rather than $Rc$.
\begin{corollary}
\label{cor:wRargmin}
Let $A,b,c,R,x^*(R)$ be defined as in~\Cref{lem:LPxR}. Then

\begin{equation}
\label{eq:LPwR}
	  \begin{array}{rcl}
	   \frac 1Rx^*(R) = w^*(R) := & \argmin_{w}  &\frac12 \norm{w-c}^2 \\ 
	    &  \text{s.t. } &Aw = \frac 1Rb\in \Rm  \\
	    &         &w \in \Rnp.
    \end{array}
\end{equation}
\end{corollary}
\begin{proof}
From 

\[
\norm{x-Rc}^2 =R^2\norm{\frac 1Rx-c}^2 =R^2\norm{w-c}^2, \, x=Rw,
\]
we substitute for $x$ in \cref{eq:LPxR} and 
obtain: $A(Rw) = b \iff Aw = \frac 1Rb$.
The result follows from the observation that $\argmin$ does not change
after discarding the constant $R^2$.
\end{proof}

\subsubsection{Warm Start; Stepping Stone External Path Following}
\label{sect:scalRb}
\index{stepping stone external path following}
We consider the scaling in~\Cref{cor:wRargmin} and recall the relation 
between the scaling for $c$ with variable $x$:

\[
x(R) = Rw(R).
\]
(To simplify notation, we ignore the optimality symbol $\left( \cdot
\right)^*$.)
The optimality conditions from~\Cref{thm:gensimplfreevrble}
for $w=w(R)$ in~\Cref{cor:wRargmin} are:

\begin{equation}
\label{eq:optcondtriplewyz}
	\begin{pmatrix} w - c - A^Ty - z \\ Aw - \frac 1Rb \\ z^Tw \end{pmatrix}
= \begin{pmatrix} 0 \\ 0 \\ 0\end{pmatrix}, \quad w,z \in \R^n_+,\, 
y \in \R^m.
\end{equation}
We conclude that

\[
\lim_{R\to \infty} \Proj_{\range(A^T)}w(R) = 0, \, 
\lim_{R\to \infty} Rw(R)= x^*, \,\, \text{the optimum of the LP}.
\]
The optimality conditions are now

\begin{equation}
\label{eq:optscaleb}
w = c+A^Ty + z, \, b = ARw = AR(c+A^Ty)_+, \quad
w^Tz = 0,\, w,z\geq 0.
\end{equation}
This means that $\|w\|$ is an estimate for the error in dual
feasibility, i.e.,~an estimate for the accuracy of $Rw$ as the optimum
of the original LP.

Given the current $R$ and the approximate optimal triplet
$(w(R),y(R),z(R))$, we
would like to find a good new $R_n\geq R$ and a corresponding $y_n$ to send to
the projection algorithm for a warm start process. We use sensitivity
analysis for the best approximation problem.
\begin{theorem}
\label{thm:pathfollowR}
Suppose  {\crb $R>0$ is given} and the triplet $(w,y,z) {\crb = (w(R),y(R),z(R))}$ 
is {\crb primal-dual} optimal for~\cref{eq:LPwR}; 
i.e.,~satisfies~\cref{eq:optcondtriplewyz}. Let 

\begin{equation}
\label{eq:defB}
\begin{array}{c}
	\textdef{$\cN =\cN(z)=  \{i\,:\, z_i > 0 \}$}, \,
\textdef{$\cB=\cB(w) = {\crb \{i \,:\, w_i>0 \}}$}, \,
\textdef{$\cZ=\cZ(w,z) = {\crb \{i \,:\, w_i=z_i=0 \}}$};\\
e = \begin{pmatrix}
   b_\cB-Rw_\cB\cr -(b_\cN+Rz_\cN)
    \end{pmatrix},\quad
f = \begin{pmatrix}
   Rb_\cB\cr -Rb_\cN
    \end{pmatrix},
\end{array}
\end{equation}
where $b_\cB, b_\cN$ are defined in \cref{eq:deltawB} and
\cref{eq:zAbN}, respectively.
Then the maximum value for increasing $R$ {\crb and maintaining
both optimality and the indices in the bases sets $\cB,\cN,\cZ$} is 

\begin{equation}
\label{eq:Rn}
R_n = \min\{{\crb f_i/e_i : e_i > 0, f_i > 0, \, \forall i}\}.
\end{equation}
The corresponding changes $\Delta w,\Delta y,\Delta z$ that result in 
$w+\Delta w,y+\Delta y,z+\Delta z$ {\crb still} optimal for $R_n$
are given in the proof {\crb in \cref{eq:deltawB}, \cref{eq:detlay},
\cref{eq:zAbN}, respectively}.

Moreover, if $R_n=\infty$, then the optimal solution of the \LP has been
found.
\end{theorem}
\begin{proof}
We {\crb first} want to find the maximum increase in $R$ that keeps the current
basis $\cB$ optimal for~\cref{eq:LPwR}, {\crb i.e.,~we maintain
\[
z_i\geq 0, \forall i\in \cN,\,
w_i\geq 0, \forall i\in \cB,\,
w_i=z_i = 0, \forall i\in \cZ.
\] 
}
To maintain the feasibility from the three basis sets in
\cref{eq:defB}, we have 
 \label{page:dwB}

\begin{equation}
\label{eq:sensanaleqns}
\begin{array}{c}
A_\cB (w_\cB+\Delta w_\cB) = \frac 1{R_n} b \implies 
         A_\cB \Delta w_\cB 
=  \left(\frac 1{R_n} - \frac 1{R}\right) b
\\ w_B +\Delta w_\cB -c_\cB -A_\cB^T(y+\Delta y)  = 0 \implies
                                    \Delta w_\cB  =A_\cB^T(\Delta y)
\implies A_\cB \Delta w_\cB = A_\cB A_\cB^T(\Delta y) = 
                \left(\frac {R-R_n}{RR_n}\right) b
\\  -c_\cZ -A_\cZ^T(y+\Delta y)   = 0 \implies
                     A_\cZ^T(\Delta y) = 0
\\  -c_\cN -A_\cN^T(y+\Delta y) - (z_\cN +\Delta z_\cN)  = 0 \implies
                    \Delta z_\cN =  -A_\cN^T(\Delta y).
\end{array}
\end{equation}
{\crb We have two equations to solve for $\Delta y$.
When strict complementarity fails,
we choose a full column rank matrix $V_\cZ$ that satisfies $\range(V_\cZ) =
\nul(A_\cZ^T)$; otherwise $V_\cZ=I$. Then we solve to get
\begin{equation}
\label{eq:detlay}
\Delta y_p := V_\cZ \left( A_\cB A_\cB^T V_\cZ \right)^\dagger b,\,
\Delta y := \left(\frac {R-R_n}{RR_n}\right) \Delta y_p.\footnote{Note
that in applications we can include indices from $\cZ$ in $\cB$. This
allows for a greater choice for $\Delta y,\Delta w_B$.}
\end{equation}
Note that a solution exists since $b\in \range(A_\cB)$.}\footnote{{\crb In 
the nondegenerate case we get a simplification
since $A_\cB^T \left(A_\cB A_\cB^T\right)^\dagger  =
A_\cB^\dagger$.}}
We now have

\begin{equation}
\label{eq:deltawB}
-w_\cB\leq {\crb \Delta w_\cB} 
= A_\cB^T \left(\frac {R-R_n}{RR_n}\right) \Delta y_p
 = -\left(\frac {R_n-R}{RR_n}\right)A_\cB^T\Delta y_p
 =: -\left(\frac {R_n-R}{RR_n}\right)b_\cB.
\end{equation}
We get that 

\begin{equation}
\label{eq:equivineqs}
(R_n-R)b_\cB \leq (RR_n)w_\cB \,\iff \, R_n(b_\cB-Rw_\cB) \leq
	Rb_\cB.
\end{equation}
To find the
maximum $R_n$ and check that it is not $R_n=\infty$, we use an \LP type
ratio test. We set the two vectors to be
\[
\textdef{$e_\cB=(b_\cB-Rw_\cB)$},\, \textdef{$f_\cB= Rb_\cB$}.
\]
Note that the inequalities in \cref{eq:equivineqs} hold trivially for 
$R_n=R$. 
{\crb For simplicity of notation, we ignore the subscript $\cB$ and use
$e,f$.}
Therefore, we
cannot have {\crb both} $e_i > 0, f_i\leq 0$. We choose $R_n$ to be the 
maximum that satisfies {\crb the ratio test, i.e.,~we get:}

\begin{equation}
\label{eq:ineqtwo}
\crb
R_n = \min_i 
\{ f_i/e_i\,:\,   f_i > 0, e_i >0 , \, i \in \cB\},
\end{equation}
where the minimum over the empty set is by definition $+\infty$. 
{\crb
Note that
$\max_i \{ f_i/e_i\,:\,   f_i < 0, e_i <0 , \, i \in \cB\} \leq R_n$
always holds since $R_n=R>0$ satisfies the
inequality. Moreover, the result simplifies in the nondegenerate case as
we  have
\[
A_\cB^T \left(\frac {R-R_n}{RR_n}\right) \Delta y_p
 = -\left(\frac {R_n-R}{RR_n}\right)A_\cB^\dagger b 
 = -\left(\frac {R_n-R}{RR_n}\right)b_\cB,\quad b_\cB = A^\dagger_\cB b.
\]
We can then set
$R_n=\infty$ if $A_\cB$ is full column rank 
or $b_\cB = w_\cB$, i.e.,~we have the (best) least squares solution.}

Similarly we now need a ratio test for $z_\cN$ to maintain dual
feasibility and nonnegativity. {\crb Note that we set $\Delta z_i=\Delta
w_i = 0, \forall i\in \cZ$.}
We have

\begin{equation}
\label{eq:zAbN}
-z_\cN\leq {\crb \Delta z_\cN} 
= -A_\cN^T \left(\frac {R-R_n}{RR_n}\right) \Delta y_p
 = \left(\frac {R_n-R}{RR_n}\right)A_\cN^T 
 \Delta y_p
 =: \left(\frac {R_n-R}{RR_n}\right)b_\cN.
\end{equation}
We get that 

\[
(R_n-R)b_\cN \geq -(RR_n)z_\cN \,\iff \, R_n(-b_\cN-Rz_\cN) \leq
	-Rb_\cN.
\]
We again find the maximum $R_n$ and check that we do not have 
$R_n=\infty$ using an \LP type
ratio test. We set the two vectors to be
$e_\cN=-(b_\cN+Rz_\cN),\, f_\cN=  -Rb_\cN$.
Recall that the inequality holds trivially for $R_n=R$. 
{\crb Again, for simplicity of notation, we ignore the subscript $\cN$ and use
$e,f$.}
Therefore, we
cannot have $e_i > 0, f_i\leq 0$. We choose $R_n$ to be the maximum that
satisfies:

\[
\max_i \{ f_i/e_i, \text{  if  }  f_i < 0, e_i <0, \, i\in \cN \} \leq
R_n = \min_i 
\{ f_i/e_i, \text{  if  }  f_i > 0, e_i >0, \, i\in \cN \}.
\]

We choose $R_n$ as the minimum of the above two values found.

Finally, if $R_n=\infty$, then the bases do not change as $R$
increases to infinity, i.e.,~the optimal bases have been found.
\end{proof}

The above~\Cref{thm:pathfollowR} illustrates the external path following
algorithm that we are using. The theorem finds specific values of $R$,
\emph{stepping stones on the path},
where the current choice of columns of $A$ changes. Once we find that
the next \textdef{stepping stone} is at infinity, we know that we have
found the optimal choice of columns of $A$. Thus, we have an external
path following algorithm with parameter $R$ but we only choose specific
points on this path to \emph{step} on. 
{\crb The algorithm is particularly efficient for nondegenerate problems,
$\cZ=\emptyset$, where the sensitivity analysis is accurate. 
For highly degenerate problems, restricting $\Delta w_i=\Delta z_i=0,
\forall i \in \cZ$, can severely restrict increasing $R$, 
see~\Cref{sect:LPappTest} below.
}

\subsubsection{Upper and Lower Bounds for the \LP Problem}
The optimal solution from the projection problems~\cref{eq:LPxR,eq:LPwR}
provides a feasible $x$, and we get the corresponding \textdef{\LP lower
bound} $c^Tx^*(R)$. The upper bound is not as easy and more important in
stopping the algorithm.

Note that in~\Cref{sect:scalRb} primal feasibility and complementary
slackness hold for $x(R) = Rw$ and $z$, and this is identical for the LP
problem. Therefore, we need to find $y_\LPt$ to satisfy the LP dual
feasibility

\[
z_\LPt = A^Ty_\LPt - c \geq 0.
\]
But, from the projection problem optimality conditions we have

\[
A^T(-y) = z + c - w, \, 0\leq  z = A^T(-y) -c + w, \, w\geq 0.
\]
 \label{page:precleq}
As seen above, this means that in the limit, $w$ is small and we do get
dual feasibility $y(R)\to y_\LPt$. But at each iteration we actually have

\begin{equation}
\label{eq:otpcondR}
 z-w = A^T(-y) -c, \, z,w\geq 0, z^Tw = 0, \quad y\cong y_R.
\end{equation}
We can write the required dual feasibility equations using the indices for $w_i>0$.

\[
A_i^Ty - c_i \in \left\{\begin{array}{rl}  \{0\}, & \text{if  } w_i > 0, \\
                                            \R_+, & \text{if  } w_i = 0.
\end{array}
\right.
\]
Recall the definitions of $\cN,\cB$ in~\cref{eq:defB}.
Then for a given $y_R$ from the optimality conditions from the projection
problem~\cref{eq:otpcondR},
we consider the nearest dual \LP feasible system with unknowns $z\geq
0,y_\LPt$. Note that we are using the projection with free
variables,~\Cref{sect:freevrbls}.
\begin{lemma}
\label{lem:upperbndLP}
Let $w,y,z$ be approximate optimal solutions from~\cref{eq:optscaleb}
and $\cB$ the support defined in~\cref{eq:defB}. Consider the following
\BAP for the given dual variables.

\begin{equation}
\label{eq:dualfeasB}
\begin{array}{rcl}
\begin{pmatrix}y^*_\LPt\cr z^*_\LPt\end{pmatrix} \in &\argmin & \frac 12
\|(-y) - y_\LPt\|^2 
+\frac 12 \|0-(z_\LPt)_\cB\|^2
+ \frac 12 \|z_\cN - (z_\LPt)_\cN\|^2 \\
& \text{s.t.} & 
\begin{bmatrix} A_\cB^T &-I & 0\cr
A_{\cN}^T   &0&-I
\end{bmatrix}
\begin{pmatrix}
y_\LPt \cr
(z_\LPt)_\cB \cr
(z_\LPt)_\cN
\end{pmatrix}
=
\begin{pmatrix}
c_\cB \cr
c_\cN
\end{pmatrix} \\ 
& & y_\LPt \text{ free}, \,
z_\LPt=
\begin{pmatrix}
(z_\LPt)_\cB\cr (z_\LPt)_\cN
\end{pmatrix}\geq 0.
\end{array}
\end{equation}
Then the optimal value of the \LP \cref{eq:PLP} satisfies the upper bound

\[
p^*_\LPt \leq b^Ty^*_\LPt.
\]
Moreover, suppose that $z_\cB =0$. Then equality holds and the \LP is
solved with primal-dual optimum pair $(w,y_\LPt)$.
\end{lemma}
\begin{proof}
Recall that the optimal value $p^*_\LPt$ is finite.
The proof of the bound follows from weak duality in linear programming.
Equality follows from the optimality conditions since primal feasibility
and complementary slackness hold with $w$.
\end{proof}

\subsection{Projection and Free Variables}
\label{sect:freevrbls}
For many applications, some of the variables are free and not all the
variables are in the objective function. We consider these two cases.
Note this can arise when the objective is a general least squares
problem, e.g.,~$\min \|Bx-c\|^2$ and we add the constraint $Bx-w=0$ and
substitute the free variable $w$ into the objective function.
\subsubsection{Projection with Free Variables}
We first consider the problem with some of the variables free:

\begin{equation}
\label{eq:projprobfree}
    \text{(P)} \qquad  
	  \begin{array}{rcl}
		  x(v) := & \argmin_{x_1,x_2}  & \frac12 \norm{x-v}^2,\quad
     x= \begin{pmatrix}x_1 \cr x_2 \end{pmatrix},\,
     v= \begin{pmatrix}v_1 \cr v_2 \end{pmatrix},
\\  
	    &  \text{s.t. } &Ax = b\in \Rm  \\
	    &         &x_1 \in \Rnop, \, x_2 \in \Rnt, \,\,
		  \\
         ~~\\
    \text{optimal value: } p_f^*(v) &= & \frac12 \norm{x(v)-v}^2,
    \end{array}
\end{equation}
    \index{optimal value, $p_f^*(v)$}
    \index{$p_f^*(v)$, optimal value}

\begin{theorem}
\label{thm:gensimplfreevrble}
Consider the \textdef{generalized simplex
best approximation problem with free variables} \cref{eq:projprobfree}.
Assume that the feasible set is nonempty.
Then the optimum $x(v)$ exists and is unique. Moreover, let
\index{squared residual function, $f_f(y)$}
\index{$f_f(y)$, squared residual function}

\begin{equation}
\label{eq:Ff(y)free}
F_f(y) := A
\begin{pmatrix} \left((v+A^Ty)_1\right)_+\cr (v+A^Ty)_2\end{pmatrix} -b, \quad \textdef{$f_f(y) = \frac 12\|F_f(y)\|^2$}.
\end{equation}
Then $F_f(y)=0$  $\iff y \in \argmin f_f(y)$, and

\begin{equation}
\label{eq:xofvfree}
x(v) =  \begin{pmatrix} \left((v+A^Ty)_1\right)_+\cr
(v+A^Ty)_2\end{pmatrix}, \,
\text{ for any root }  F_f(y) = 0.
\end{equation}
Let $\textdef{$p_f^*(v)$} = \frac 12 \|x(v)-v\|^2$ denote 
the \textdef{primal optimal value}. Then strong duality holds and
the \textdef{dual problem} of~\cref{eq:projprobfree} is the maximization
of the \textdef{dual functional, $\phi_f(y,z_1)$}:

\[
p_f^*(v)=\textdef{$d_f^*(v)$} := \max_{z_1 \in \R_+^{n1} , y\in \Rm} \phi(y,z_1) :=  
-\frac12 \norm{ \begin{pmatrix}z_1 \cr 0 \end{pmatrix}
- A^Ty }^2 + y^T(Av - b)-z_1^Tv_1.
\]
\index{dual functional, $\phi(y,z)$}
\index{$\phi(y,z)$, dual functional}
\end{theorem}
\begin{proof}
We modify the proof of~\Cref{thm:projoptndual}.
The \textdef{Lagrangian, $L_f(x,y,z)$} for~\cref{eq:projprobfree} is
\index{$L_f(x,y,z)$, Lagrangian}

\begin{equation} \label{projprobafffree}
 	 L_f(x,y,z) = \frac12 \norm{x-v}^2 + y^T(b-Ax)-z_1^Tx_1, \quad
 	 \nabla_x L_f(x,y,z) = x-v - A^Ty -
     \begin{pmatrix}z_1 \cr 0 \end{pmatrix}.
\end{equation}
Solving for a stationary point means

\[
0=\nabla_x L_f(x,y,z) \implies  x = v + A^Ty +z,\quad
     z = \begin{pmatrix}z_1 \cr 0 \end{pmatrix}.
\]
Therefore, with this definition of $z$, we still have
at a stationary point that

\[
\begin{array}{rcl}
 L_f(x,y,z) 
&=&
 \frac12 \norm{v + A^Ty +z-v}^2 + y^T(b-A(v + A^Ty +z) )-z^T(v + A^Ty +z)
\\&=&
 \frac12 \norm{ A^Ty +z}^2 + y^Tb -y^TAv -(A^Ty)^T(A^Ty+z) 
-z^Tv - z^T(A^Ty +z)
\\&=&
 \frac12 \norm{ A^Ty +z}^2 + y^Tb -y^TAv -(A^Ty+z)^T(A^Ty+z) -z^Tv
\\&=&
 -\frac12 \norm{ z+ A^Ty }^2 + y^T(b-Av)-z^Tv.
\end{array}
\]

As in~\Cref{thm:projoptndual}, the problem \cref{eq:projprobfree} 
is a projection onto a nonempty polyhedral set,
a closed and convex set. The optimum exists and is unique
and strong duality holds, i.e.,~there is a zero duality gap
$p_f^*=d_f^*$, and the dual value is attained. The Lagrangian dual is 

\[
\begin{array}{rcll}
d^*
&=&
    \max_{z_1 \in \R^{n_1}_+, y} \min_{x} & L_f(x,y,z) = \frac12 \norm{x - v}^2 +
    y^T(b-Ax ) - z_1^Tx_1
\\&=&
    \max_{z_1 \in \R^{n_1}_+, y,x} &  \{L_f(x,y,z_1) : \nabla_x L_f(x,y,z_1) = 0\}
\\&=&
    \max_{z_1 \in \R^{n_1}_+, y,x} &  \{L_f(x,y,z) : x = v + A^Ty +z\}
\\&=&
    \max_{z_1 \in \R^{n_1}_+, y} &  -\frac12 \norm{ z+ A^Ty }^2 + y^T(b-Av)-z^Tv.
\end{array}
\]
Therefore, we derive the \textdef{KKT optimality conditions} 
for the primal dual variables $(x,y,z)$ with
     $z = \begin{pmatrix}z_1 \cr 0 \end{pmatrix}, x_1\geq 0, z_1\geq 0$,
as follows

\begin{equation*}
    \begin{array}{ll}
        \nabla_x  L_f(x,y,z) = x - v - A^Ty - z  = 0,
& \text{(dual feasibility)} \\ 
        \nabla_y  L_f(x,y,z) = Ax - b  = 0, & \text{(primal feasibility)}\\ 
    \nabla_z  L_f(x,y,z) \cong  x  \in (\Rnp-z)^+.  &
\text{(complementary slackness $z_1^Tx_1=0$)}
    \end{array}
\end{equation*}

The standard KKT optimality conditions for primal-dual variables $(x,y,z)$
can be rewritten as:

\[
	\begin{pmatrix} x - v - A^Ty - z \\ Ax - b \\ z^Tx \end{pmatrix}
= \begin{pmatrix} 0 \\ 0 \\ 0\end{pmatrix}, \quad x_1,z_1 \in \Rnop, y \in
\R^m,\,
     z = \begin{pmatrix}z_1 \cr 0 \end{pmatrix}.
\]
Note $v+A^Ty = x-z = x+(-z)$. Therefore this is a Moreau decomposition of $v+A^Ty$, 
with $x^Tz=0, x,z\in \Rnp$, $x = (v+A^Ty)_+$. Therefore, we
get $A(v+A^Ty)_+ = b$, where we modify the definition of $_+$ 
so that we project only the first part corresponding to $x_1$ onto the
nonnegative orthant $\Rnop$ and then
this means $z_1=-\left((v+A^Ty)_1\right)_-$.

We see that the optimality conditions

\[
A
\begin{pmatrix} \left((v+A^Ty)_1\right)_+ \cr (v+A^Ty)_2 \end{pmatrix}= b,\, 
x_1= \left((v+A^Ty)_1\right)_+, x_2= (v+A^Ty)_2
\]
imply that
\[
  z=-(v+A^Ty)_-,\, z^Tx=0, x,\,z\in \Rnp\,,
	x - v - A^Ty - z =0,
\]
i.e.,~$F_f(y)=0$, for some $y\in \Rm$.
\end{proof}
For a vertex, a basic feasible solution, 
we need $n$ active constraints. The equality
constraints $Ax=b$ account for $m$, leaving $n-m$ to choose among
$1,2,\ldots,n_1$, the constrained variables in $x_1$. This leaves

\[
m_1= n_1-(n-m) =m-(n-n_1)= m-n_2 \implies \textdef{$m_1 = m-n_2$},
\text{  basic variables}.
\]

\subsection{Triangle Inequalities}
We can obtain an efficient projection onto a large set of triangle
inequalities that arise as cuts in graph problems, 
e.g.,~\cite{Piccialli_2022}.
We let $G=(V,E)$ denote a graph with vertex set $V$ and edge set $E$,
and define the sets:

\[
\TT  := \{(u,v,w) : u<v<w\in V \},
\]
and the corresponding \textdef{triangle inequalities}, where the weight
vector $x=\left(x_{uv}\right)_{uv\in E}$ here has two indices for the edge $uv$
connecting vertices $u,v$,

\begin{equation}
\label{eq:Tineq}
(I) \quad \left\{
\begin{array}{c}
\begin{array}{rcl}
x_{vw}-x_{uv}-x_{uw} &\leq& 0 \\
x_{uw}-x_{uv}-x_{vw} &\leq& 0 \\
x_{uv}-x_{vw}-x_{uw} &\leq& 0 \\
\, \forall (u,v,w) \in  \TT
\end{array} \\
0\leq x_{uv} \leq 1, \, \forall (u,v) \in E
\end{array}\right\}.
\end{equation}

We could rewrite this as a standard feasibility-seeking problem or as a best
approximation problem,  i.e.,~given an $\bar x$ we want to find the
nearest point to $\bar x$ that is in  a subset of triangle
inequalities defined by the matrix
$T$, namely with slacks $s,t$ and $e$ the vector of ones,
\index{$e$, vector of ones}
\index{vector of ones, $e$}

\[
\min \frac 12 \|x-\bar x\|^2 \text{  s.t.  } Tx+s = 0, x+t =e,\,\, 
x,t \geq 0, \, s\geq 0.
\]

We generated and solved random problems. The algorithm was very
efficient though we do not report the details here.

\section{Numerics} 
\label{sect:numerics}

\index{\RNNM, regularized nonsmooth Newton method}
\index{regularized nonsmooth Newton method, \RNNM}

In this section we compare the Regularized Nonsmooth Newton Method,
(\RNNMp), (exact and inexact) with the \HLWB method \cite{1996_Ba}
described in~\Cref{sec:BAP}, MATLAB's \emph{lsqlin} interior point
solver, {\crb and the 
\emph{quadratic programming proximal augmented
Lagrangian method}, (\QPPAL) \cite{LiangToh22}.} 
 \label{page:lsqlin}Recall our \BAPp, \Cref{eq:projproblemhyp}, 
and the pseudocode for
\HLWB in \Cref{alg:extHLWB} in \Cref{app:pseudos}. We show in our
experiments that \RNNM (exact) significantly outperforms the other methods.
These experiments are performed with an i7-4930k @ 3.2GHz, 16 GBs of RAM, and
MATLAB 2022b software.
\index{quadratic programming proximal augmented Lagrangian method, \QPPALp}
\index{\QPPALp, quadratic programming proximal augmented Lagrangian method}

Before {\crb comparing} the differences in performance of the algorithms {\crb we are experimenting with},
we elaborate on {\crb our implementation of} the \HLWB method, see also~\Cref{sec:BAP}.
\HLWB projects onto individual convex sets {\crb and computes
the next iterate}, $x^{k+1}$, by taking a specific convex {\crb combination. This combination
is determined} by a sequence of steering parameters, 
as defined in~\Cref{def:Steering-params}, and the initial point $v$, commonly referred to as the
anchor point in~\Cref{prob:BAP}. Traditionally, each projection 
is called an
\textdef{iteration}, and the collection of these iterations
is defined as a \textdef{sweep}~\cite{MR3616647}. In the
context of problem \Cref{eq:projproblemhyp}, \HLWB is iterating onto
one of the hyperplanes (sets) defined by the rows of $A$, 
denoted $a^{i_k}$, as well as the nonnegative
orthant. {\crb We complete a sweep once we project onto all the hyperplanes and onto the nonnegative orthant.}
(See steps \ref{line:A3startIf}-\ref{line:A3endIf} of \Cref{alg:extHLWB}.) 
Thus, we relate one sweep of \HLWB with one iteration of \RNNM.

\subsection{Time Complexity}
\label{sect:timecompl}
Since \RNNM is a second-order method and \HLWB is a first-order method,
we now discuss theoretical time complexity differences. 
From the  \RNNM algorithm,  \Cref{alg:ExactNonsmoothNewton}, we can see that worst-case time complexity is $O(m^3 + m^2n)$ \footnote{See \Cref{alg:ExactNonsmoothNewton} lines \ref{line:StartCompA1}-\ref{line:EndCompA1},
the total time complexity respectively is: $m^2n + m^2 + m^3 + n + 2n +
mn +2n + mn + n + m + 1 =m^2n +  m^3 + m^2 + 2mn + 5n + m + 1 = O(m^3 +
m^2n)$.} flops, of which every step but solving the linear system is efficiently parallelizable. 
It is worth mentioning that in line \ref{line:solvepd} of \Cref{alg:ExactNonsmoothNewton}, the linear system we are solving is positive definite and sparse. 
Therefore, it can be solved efficiently using the Cholesky decomposition.
From the \HLWB algorithm, \Cref{alg:extHLWB}, we can see that worst-case time complexity per iteration is  $O(mn)$ and per sweep is $ O(m^2n)$, of which every step is efficiently parallelizable. 
\footnote{See \Cref{alg:extHLWB} lines
\ref{line:StartCompA3}-\ref{line:EndCompA3}; the total time complexity respectively per iteration that projects
onto a half space is $(2n + 2) + 1 + (n + 2) + (mn + m + 1) = mn + 3n +
m + 6 = O(mn) $ flops.
Similarly, the total time complexity respectively per iteration that projects onto the nonnegative orthant is: $n + 1 + (n + 2) + (mn + m + 1) = mn + 2n + m + 4 = O(mn) $  flops of which  all flops are efficiently parallelizable. 
Therefore, in terms of sweeps the \HLWB method computes $m(mn + 3n + m + 6) +  mn + 2n + m + 4 = m^2n + 4mn + m^2 + 2n + 7m + 4 = O(m^2n) $ flops.}

From the perspective of theoretical time complexity it would be easy to
assume that \HLWB is the preferable algorithm as each of it's iterations are composed of operations that are completely parallelizable and each first-order sweep has an overall lower time-complexity. However, without performing numerical tests with varying parameters $m$ and $n$, we cannot yet conclude how a first-order method compares to a second-order method in terms of desired performance, especially as $m$ and $n$ get extremely large as observed in practice.

\subsection{Comparison of Algorithms}
When performing our numerical experiments, we refer to the discussion on techniques for comparisons
of algorithms given in \cite{MR3719098}.
In particular, we include performance profiles \cite{MR1875515}, and\label{page:enexact}
tables of the performances for \RNNM (exact and inexact), {\crb \HLWB, \emph{lsqlin}, and for QPPAL}. 

We compare the \HLWB algorithm to \RNNM by generating
a test problem {\crb with the form specified in \Cref{eq:projproblemhyp}. In this test problem,  the anchor 
$v$ lies in the relative
interior of the normal cone (negative of the polar cone)
of a vertex of the feasible polyhedron. Therefore, the vertex is the closest point to $v$.
Additionally, to ensure meaningful comparisons, we set $\norm{A} = 1$ and $\norm{v} = 1$ as no convergence results for \RNNM solving \Cref{eq:projproblemhyp} have been proven, as far as we know.}

The \RNNM algorithm starts with initializing  $x_0 \gets (v + A^Ty_0)_+$, where either $y_0 = 0_m$ or we are given
a $y_0$ for a warm start {\crb (as discussed in our LP application). Then,}  $x_0 \gets (v + A^Ty_0)_+$ reduces to $x_0 \gets \max(v,0)$ 
in the initialization stage of \RNNMp. Therefore, to ensure all algorithms start at the same point, we initialize $x_0 \gets  \max(v,0)$
for \HLWB, and provide $x_0 \gets  \max(v,0)$ as a warm start for
MATLAB's \emph{lsqlin} solver. {\crb Since \QPPAL performs an \ADMM warm-start, there is no way to provide a warm start point for it}.

Since \RNNM solves a reduced KKT
condition for a convex problem, the term $\frac{\norm{F(y_k)}}{1 +
\norm{b}}$ is a sufficient relative residual to serve as a 
stopping condition for \RNNMp.
Since \HLWB is a first order method, {\crb its stopping criterion is
measured at the end of a sweep, rather than at} the end of an iteration.
Furthermore, \HLWB does not have any proper stopping criterion, but 
converges in the limit{\crb. Therefore, we use the relative primal feasibility
residual,
i.e., $\frac{\norm{A\hat x_k-b}}{1 + \norm{b}}$, as the stopping criterion.} 
Note that we use $y_k$ instead of
$x_k$ in the stopping criterion as $\hat x_k$ is nonnegative 
at the end of every sweep. The \emph{lsqlin} solver uses first-order 
optimality conditions. {\crb As in \emph{lsqlin}, \QPPAL uses
first-order optimality conditions, and we report the relative optimality
gap, $\displaystyle |p^* - d^*|/\left(1 + (|p^*| + |d^*|)/2\right)$ for the relative residual
of \QPPAL.
Before discussing the generation of the problems, it is
worth noting that we are choosing to use \QPPAL's
Cholesky decomposition direct solver instead of its inexact solver. 
In addition, we increase the maximum number of iterations for the two phases of \QPPAL to match the maximum number of sweeps the other methods utilize. Furthermore, we inform \QPPAL that the
quadratic has $Q = I$, the identity.}

In \Cref{sect:nondegvertex}, we generate problems such that $v$ lies in the relative  interior of the normal cone of a
nondegenerate vertex. We also experiment with degenerate vertices, but
observe very similar results.
These tests, and the performance of
the \RNNM algorithm {\crb help to motivate} the theory
and potential practice of using \RNNM for LP applications, as seen in \Cref{sect:LPappTest}.

For the performance profiles in \Cref{sect:nondegvertex},
we use the following notation from
\cite{MR3719098}. Let $P$ { \crb denote our set of problems with varying $m$, $n$, and density. Similarily, let $S$ represent our set of solvers,
\RNNM  (exact and inexact), \HLWB, \emph{lsqlin}, and \QPPAL}. We define the
performance measure $t_{p,s} > 0$ {\crb for each pair $(p,s) \in P
\times S$ as the computational time of solver $s$
to solve problem $p$}. For each problem $p \in P$ and solver $s \in
S$, we define the performance ratio as

\[
	r_{p,s} = 
	\begin{cases}
\frac{t_{p,s}}{\min\{t_{p,s} \, : \, s \in S\}}, & \text{if convergence test passed}, \\ 
\infty,	   & \text{if convergence test failed}.
	\end{cases}
\]
{\crb The solver $s$ that performs the best on problem $p$ will have
a performance ratio of $1$. Solvers that perform worse than $s$
on problem $p$ will satisfy $t_{p,s} > 1$. In other words, the larger the performance ratio, the worse the solver performed on problem $p$}.

 The performance profile of a solver $s$ is defined as

\[
	\rho_s(\tau) = \frac{1}{|P|} \text{ size} \{p \in P \, : \, r_{p,s} \leq \tau\}.
\]
Therefore, $\rho_s(\tau)$ {\crb represents} the relative portion of time {\crb in which} the performance ratio $r_{p,s}$ for solver $s$ is within a factor $\tau \in \R$ of the best possible performance ratio.

\subsubsection{Numerical Comparisons}
\label{sect:nondegvertex}
We tested the algorithms with optimal solutions at:
nondegenerate vertices, degenerate vertices and non-vertices. They all exhibited
similar results. Therefore, we present results restricted to nondegenerate
vertices.
We begin with choosing $v$ for \Cref{eq:projproblemhyp} such
that the optimum is uniquely a nondegenerate vertex of $P$. In the
tables below we vary {\crb $m$, $n$,} and the problem density to
illustrate the changes
in each solver's performance. A data point in each table is the arithmetic mean of 5 randomly generated problems of the specified parameters that also satisfy $\norm{A} = 1, \, \norm{v} = 0.1$. For example, the first row of \Cref{table: nondegenerate sizeM3}
represents a problem with parameters $m = 500, \, n = 3000$, and a density of $0.0081$, and each solver will solve 5 randomly generated problems of the form discussed in \Cref{eq:projproblemhyp}, and the average time and relative residual from solving all 5 problems is displayed in the table. The desired stopping tolerance for the tables and performance profiles is $\varepsilon = 10^{-14}$ and maximum iterations (sweeps) is $2000$ for all solvers. {\crb Lastly, it should be noted that the regularization parameter of \RNNM for these experiments is chosen in an adaptive way. It takes into account the relative residual as defined in line \ref{line:stopcrit} of \Cref{alg:ExactNonsmoothNewton}, the norm of the Newton direction, and the norm of $v$. The purpose of this is to decrease the amount of regularization as we approach the optimal solution while accounting for the norms of the Newton direction and $v$. This regularization parameter is explicitly defined as

\begin{equation}
\begin{array}{c}
\label{eq:reg_param}
	\lambda_{k+1} = \mean \left(\left(10^{-2} F_k\right) \max(1,
	\log_{10}(\|d_k\|)), \left(10^{-3}  F_k\right) \max(1, \log_{10}(\|v\|)), 10^{-3} F_k \right),
\end{array}
\end{equation}
where $F_k$ is the relative residual at iteration $k$, and $d_k$ is the
Newton direction.
}

From \Cref{table: nondegenerate sizeM3,table: nondegenerate sizeN3,table:
nondegeneratedense3}, the empirical evidence demonstrates the
superiority of the \RNNM (exact) approach over the other solvers.
Since the \RNNMp's reduced KKT system is $m \times m$ and solved using the Cholesky Decomposition, it's performance should be affected most noticeabley as $m$ varies or density increases. 
This theoretical observation can be seen in \Cref{table: nondegenerate sizeM3,table: nondegenerate sizeN3,table:
nondegeneratedense3}, as the \RNNM (exact and inexact) algorithm is slower to converge for increasing $m$ and density,

\label{page:seen}
but is not affected by an increase in $n$.

From  \Cref{fig: nondegenerate perprofilemn}  the empirical evidence shows similar results to the tables, but better demonstrates the differences in performance between 
 \RNNM (exact) and the other solvers. The problems in \Cref{subfig:
sizeM3} are similar to those of \Cref{table: nondegenerate sizeM3} except $m$ varies by $100$ from $100$ to $2000$. 
Similarly, the problems in \Cref{subfig: sizeN3} have $n$ varying by $100$ from $3000$ to $5000$, and  \Cref{subfig: dense3} has density varying by 1\% from 1\% to 100\%. 
\label{page:have}In every performance profile, the
\RNNM (exact) algorithm clearly
{\crb outperforms the other solvers in our experiments}, with  \RNNM (inexact) performing well for an inexact method on mid-sized problems. 
{\crb Conversely}, \HLWB is relatively slow on these problems. {\crb This can be attributed to its linear convergence rate. Due to it's linear convergence, it will perform a large number of sweeps, which can amount to millions of iterations on certain problems with large $m$.}
Performance profiles can be found in \Cref{subsect: nondeg perf} with
the stopping tolerances $\varepsilon = 10^{-2}, 10^{-4}$, to illustrate
that \RNNM (exact) outperforms {\crb HLWB and \emph{lsqlin} at different tolerances, but \QPPAL remains competitive}. 
\begin{table}[H]
{\small
\flushleft
\caption{Varying problem sizes $m$; comparing computation time and
relative residuals.}
\scalebox{0.8}{
\begin{tabular}{|ccc||ccccc|ccccc|} \hline
\multicolumn{3}{|c||}{Specifications} & \multicolumn{5}{|c|}{Time (s)} & \multicolumn{5}{|c|}{Rel. Resids.}\cr\cline{1-13}
  $m$&   $n$& \% density&  Exact & Inexact &    HLWB &  lsqlin &   QPPAL &   Exact & Inexact &    HLWB &  lsqlin &   QPPAL \cr\hline
  500 &  3000 & 8.1e-01 &4.23e-02 & 1.51e-01 & 1.54e+02 & 3.77e+00 & 1.14e+00 &  1.96e-16 & 8.26e-16 & 2.25e-04 & 7.26e-17 & 1.72e-17 \cr\hline
 1000 &  3000 & 8.1e-01 &4.40e-01 & 9.97e-01 & 3.71e+02 & 5.37e+00 & 2.15e+00 &  2.70e-16 & 1.95e-15 & 2.14e-04 & 3.87e-17 & 2.70e-17 \cr\hline
 1500 &  3000 & 8.1e-01 &1.17e+00 & 3.23e+00 & 6.09e+02 & 7.02e+00 & 4.69e+00 &  3.41e-17 & 6.73e-16 & 2.27e-04 & 3.95e-17 & 1.16e-17 \cr\hline
 2000 &  3000 & 8.1e-01 &2.49e+00 & 7.51e+00 & 8.67e+02 & 1.02e+01 & 7.81e+00 &  6.11e-17 & 3.11e-17 & 2.24e-04 & 3.14e-17 & -2.74e-17 \cr\hline
\end{tabular}

}
\label{table: nondegenerate sizeM3}
}
\end{table}

\begin{table}[H]
{\small
\flushleft
\caption{Varying problem sizes $n$; comparing computation time and
relative residuals.}
\scalebox{0.8}{
\begin{tabular}{|ccc||ccccc|ccccc|} \hline
\multicolumn{3}{|c||}{Specifications} & \multicolumn{5}{|c|}{Time (s)} & \multicolumn{5}{|c|}{Rel. Resids.}\cr\cline{1-13}
  $m$&   $n$& \% density&  Exact & Inexact &    HLWB &  lsqlin &   QPPAL &   Exact & Inexact &    HLWB &  lsqlin &   QPPAL \cr\hline
  200 &  3000 & 8.1e-01 &3.12e-03 & 3.69e-02 & 4.45e+01 & 3.50e+00 & 8.66e-01 &  8.64e-18 & 7.39e-17 & 2.56e-04 & 6.52e-16 & 5.89e-17 \cr\hline
  200 &  3500 & 8.1e-01 &3.08e-03 & 4.05e-02 & 5.17e+01 & 4.93e+00 & 1.00e+00 &  9.07e-18 & 1.26e-17 & 2.78e-04 & 1.23e-15 & 2.15e-17 \cr\hline
  200 &  4000 & 8.1e-01 &3.24e-03 & 3.70e-02 & 5.82e+01 & 7.31e+00 & 1.09e+00 &  1.46e-16 & 8.91e-16 & 2.80e-04 & 3.21e-16 & -9.18e-18 \cr\hline
  200 &  4500 & 8.1e-01 &3.99e-03 & 4.17e-02 & 6.58e+01 & 1.01e+01 & 1.18e+00 &  1.80e-15 & 2.05e-16 & 3.13e-04 & 4.61e-17 & 1.71e-16 \cr\hline
\end{tabular}

}
\label{table: nondegenerate sizeN3}
}
\end{table}

\vspace{.1in}

\begin{table}[H]
{\small
\flushleft
\caption{Varying problem density; comparing computation time and relative residuals.}
\scalebox{0.8}{
\begin{tabular}{|ccc||ccccc|ccccc|} \hline
\multicolumn{3}{|c||}{Specifications} & \multicolumn{5}{|c|}{Time (s)} & \multicolumn{5}{|c|}{Rel. Resids.}\cr\cline{1-13}
  $m$&   $n$& \% density&  Exact & Inexact &    HLWB &  lsqlin &   QPPAL &   Exact & Inexact &    HLWB &  lsqlin &   QPPAL \cr\hline
  300 &  1000 &   25 &5.69e-02 & 2.66e-01 & 4.55e+01 & 3.30e-01 & 1.20e+00 &  2.83e-17 & 1.14e-17 & 1.50e-04 & 8.61e-17 & 5.99e-17 \cr\hline
  300 &  1000 &   50 &5.43e-02 & 2.28e-01 & 5.39e+01 & 3.08e-01 & 1.82e+00 &  1.23e-16 & 1.97e-17 & 1.44e-04 & 8.08e-16 & 1.42e-17 \cr\hline
  300 &  1000 &   75 &7.75e-02 & 2.86e-01 & 5.36e+01 & 3.16e-01 & 1.49e+01 &  4.83e-16 & 1.72e-17 & 1.62e-04 & 3.49e-16 & -3.43e-16 \cr\hline
  300 &  1000 &  100 &7.27e-02 & 2.47e-01 & 4.65e+01 & 3.00e-01 & 2.54e+02 &  5.66e-16 & 2.15e-17 & 1.63e-04 & 1.91e-15 & 1.04e-14 \cr\hline
\end{tabular}

}
\label{table: nondegeneratedense3}
}
\end{table}

\begin{figure}[H]

\centering

\begin{subfigure}[t]{0.45\linewidth}
	\centering
	\includegraphics[width = 0.95\linewidth, height =
	0.25\textheight, keepaspectratio]{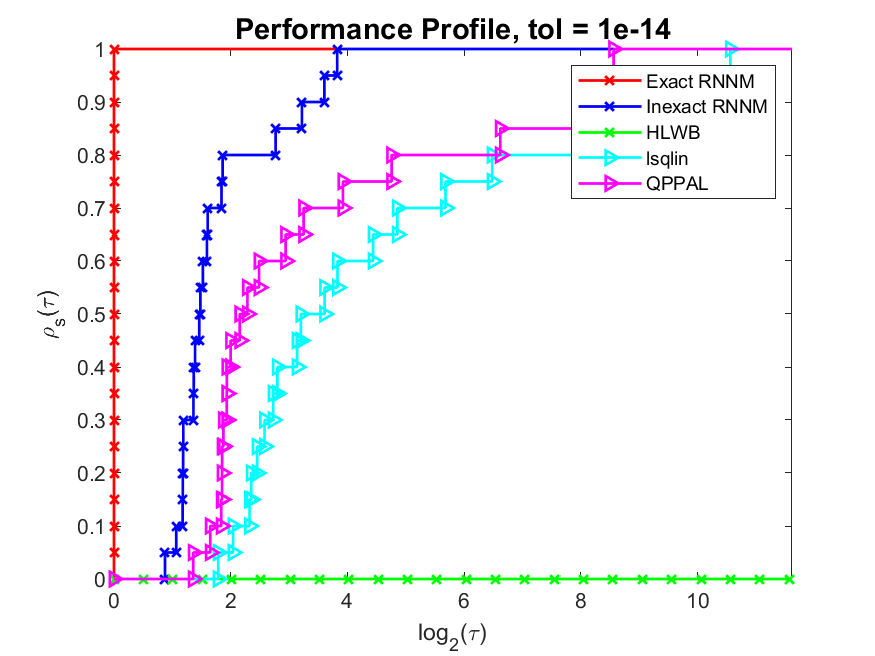}
	\caption{Varying problem sizes $m$.}
	\label{subfig: sizeM3}
\end{subfigure}
\begin{subfigure}[t]{0.45\linewidth}
	\centering
	\includegraphics[width = 0.95\linewidth, height =
	0.25\textheight, keepaspectratio]{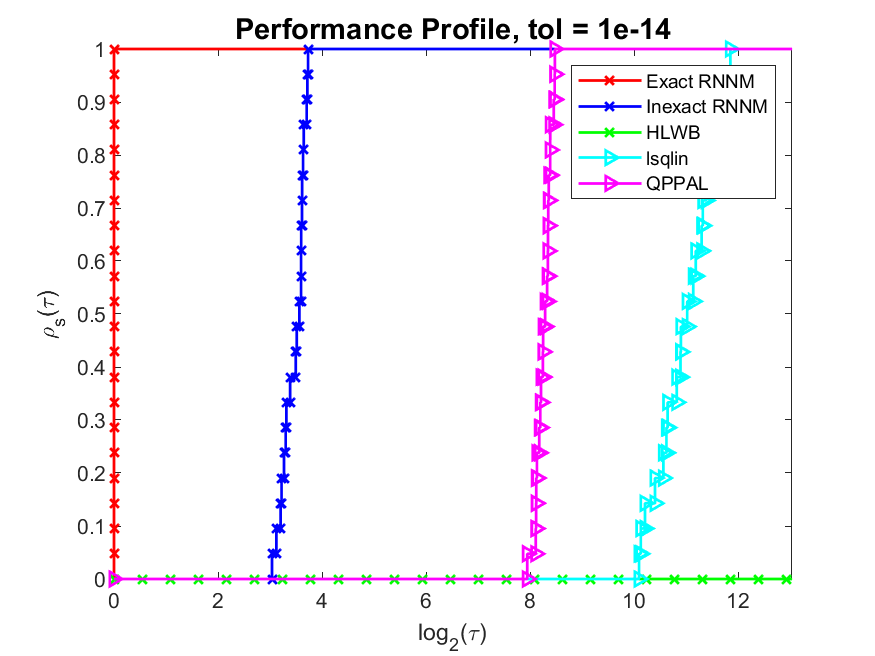}
	\caption{Varying problem sizes $n$.}
	\label{subfig: sizeN3}
\end{subfigure}
\begin{subfigure}[t]{0.45\linewidth}
	\centering
	\includegraphics[width = 0.95\linewidth, height =
	0.25\textheight, keepaspectratio]{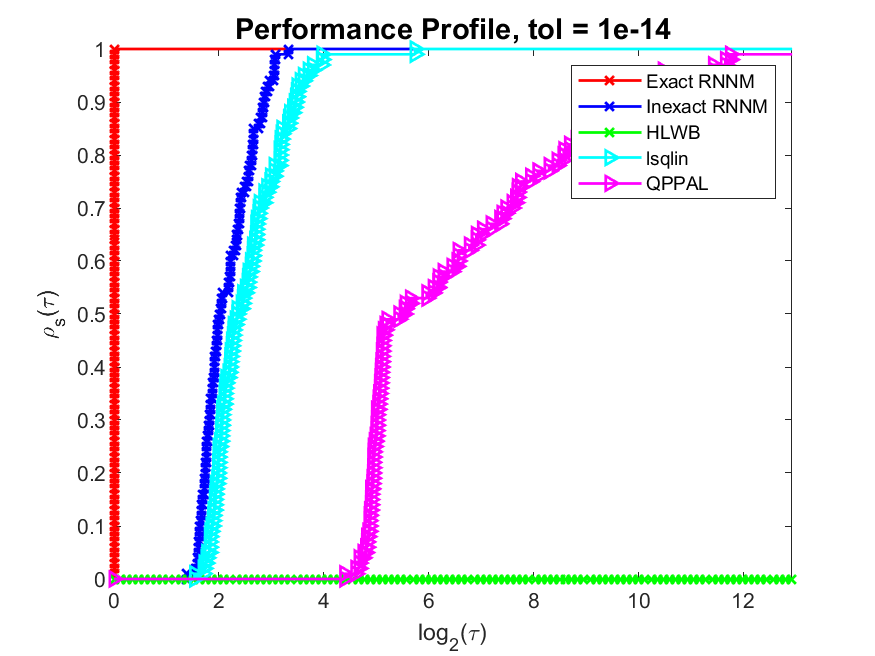}
	\caption{Varying problem density.}
	\label{subfig: dense3}
\end{subfigure}
\caption{Performance profiles for problems with varying $m$, $n$, and densities for nondegenerate vertex solutions.}
\label{fig: nondegenerate perprofilemn}
\end{figure}

\subsection{Solving Large Sparse Linear Programs} 
\label{sect:LPappTest}
We now apply~\Cref{eq:LPxR} and~\Cref{thm:pathfollowR} 
to solve {\crb large-scale randomly generated \LPp s, and problems from
the NETLIB dataset}.  {\crb We call this method the \emph{stepping
stones external path following algorithm}, (SSEPF)}, and note
that we use the estimate for a starting $R$ given
in~\cref{eq:Restimate}. \index{ stepping stones external path following algorithm,SSEPF} The stepping stones are found using $R_n$ in~\cref{eq:Rn}. We add a small decreasing scalar to $R_n$ to ensure
that we {\crb change the basis of $A$ at each iteration}.
{\crb For simplicity, we restrict ourselves to
nondegenerate \LPp s for the randomly generated problems}.

We compare SSEPF with the MATLAB \emph{linprog} code, using both the dual simplex and the
interior-point algorithms. {\crb We also compare with Mosek's dual
simplex and interior point method, and with the
\emph{semismooth Newton inexact proximal augmented Lagrangian method}, 
(SNIPAL) \cite{MR4146384}}. 
We use randomly generated problems scaled so
that $\norm{A}=1$, and the optimal solution $x^*$ satisfies $\norm{x^*}=1$.
A data point in \Cref{table:LPAppl} is the
arithmetic mean of $5$ randomly generated problems of the specified
parameters. We exclude instances
where a {\crb method fails to provide a solution from \Cref{table:LPAppl} for clarity, but these instances are plotted in \Cref{fig: LPperf} as a failure to converge}.
{\crb Since the smallest stopping tolerance allowed by \emph{linprog} is $\varepsilon = 10^{-10}$, a linear program is considered successfully solved in the performance profile of \Cref{fig: LPperf}  if the optimality gap is less than or equal to  $\varepsilon = 10^{-8}$}. The maximum number of iterations
for \emph{linprog} and Mosek is the default number, {\crb and for SNIPAL
it is $2000$}. The relative residual shown {\crb in} \Cref{table:LPAppl} is the
sum of {\crb the} relative primal feasibility, dual feasibility, and
complementary slackness. In other words, let $(x^*,y^*,z^*)$ be the
optimal solution {\crb an algorithm returns}, then the relative residual
as shown in the table is
\index{semismooth Newton inexact proximal augmented Lagrangian method, SNIPAL} 

\[
	\frac{\norm{Ax^* - b}}{1 + \norm{b}} + \frac{\norm{z^* - A^Ty^*
+ c}}{1 + \norm{c}} + \frac{(x^*)^Tz^*}{1 + \max(\norm{x^*},\norm{z^*})}.
\]

{\crb When discussing the performance of SSEPF, it should be noted that
we are using the exact \RNNM direction to solve the BAP subproblem, and using
\cref{eq:reg_param} to compute the regularization parameter. We denote
this in \Cref{table:LPAppl} and  \Cref{fig: LPperf} as SSEPF-RNNM}.
{\crb Furthermore, we use the abbreviations Linprog DS and Linprog IPM to refer to \emph{linprog's} dual simplex and
interior point method, respectively. Likewise, 
we use similar abbreviations for Mosek}.

From \Cref{table:LPAppl}, the empirical evidence demonstrates that the
stepping stone approach performs better than {\crb MATLAB's dual simplex
and interior point method on most problems, and has proven to be quite
competitive with Mosek's dual simplex and interior point method}. This
becomes more evident as

\label{page:grow}the {\crb sizes} of the problems grow and the problems become sparser. {\crb In other words, }
we see that our code fully exploits sparsity in \LPp. {\crb This can be seen when observing the performance of SSEPF-RNNM with respect to time on the rows of \Cref{table:LPAppl} where the problem density decreases. Despite the increase in problem dimension, the decrease in density leads to an increase in performance in comparison to the previous row}.
 {\crb Another thing to} notice {\crb is} that in rows $5$-$9$ of \Cref{table:LPAppl},
{\crb \emph{linprog's}} interior point method {\crb and Mosek's dual simplex method} failed to converge to a solution 
after having reached the default maximum number of iterations. 

In \Cref{sect:nondegvertex}, the performance profiles were constructed by looking at smaller intervals of varying $m,n$ and density.
For example \Cref{table: nondegenerate sizeM3} shows results where $m$
varies by increments of $500$, but in \Cref{subfig: sizeM3} $m$ varies
by increments of $100$.
Since \emph{linprog's} interior point method {\crb and Mosek's dual simplex method} struggled with obtaining {\crb the desired}
primal feasibility, as seen in \Cref{table:LPAppl}, \Cref{fig: LPperf} shows 
the performance of each solver with respect to all 50 problems instead of examining the average performance. 

It is important to note that the performance profile exhibits more
failed solutions from the dual simplex and interior point methods of
MATLAB. We have tried taking the maximum of the primal feasibility, dual feasibility, and complementary slackness returned by 
MATLAB's \emph{linprog} function instead of the sum, and both revealed
equivalent results. In other words, we are not sure why there are more
problems failing at this tolerance than reported by MATLAB, but it further distinguishes our
stepping stone approach from MATLAB's \emph{linprog} algorithms. {\crb
Mosek, and more specifically Mosek's interior point method is very
competitive, as  \Cref{fig: LPperf} shows. {\crb Unfortunately, SNIPAL failed to converge on every problem in this dataset.
We have seen it converge successfully on some random linear programming problems, but none of the ones that we
generated in our Numerical Experiments section}. It is worth noting that the
table which shows the average performance of $5$ randomly generated
problems with respect to a set of parameters indicates that SSEPF-\RNNM
performs better than Mosek's interior point method in $7$ out of $10$ rows
in the table.}

\begin{table}[H]
{\small
\flushleft
\scalebox{0.55}{
\begin{tabular}{|ccc||cccccc|cccccc|} \hline
\multicolumn{3}{|c||}{Specifications} & \multicolumn{6}{|c|}{Time (s)} & \multicolumn{6}{|c|}{Rel. Resids.}\cr\cline{1-15}
  $m$&   $n$& \% density&   SSEPF-RNNM & Linprog DS & Linprog IPM & MOSEK DS & MOSEK IPM &  SNIPAL &    SSEPF-RNNM & Linprog DS & Linprog IPM & MOSEK DS & MOSEK IPM &  SNIPAL \cr\hline
2e+03 & 5e+03 & 1.0e-01 &8.94e-02 & 3.09e-02 & 4.50e-02 & 1.46e-01 & 1.64e-01 & 6.90e+00 & 3.38e-17 & 2.63e-16 & 4.88e-09 & 1.31e-16 & 1.53e-16 & 2.14e-04 \cr\hline
2e+03 & 1e+04 & 1.0e-01 &9.64e-02 & 4.84e-02 & 7.53e-02 & 1.49e-01 & 1.93e-01 & 8.31e+00 & 2.82e-17 & 6.00e-16 & 1.60e-04 & 1.31e-16 & 2.89e-16 & 1.72e-04 \cr\hline
2e+03 & 1e+05 & 1.0e-01 &1.68e-01 & 3.91e-01 & 7.45e-01 & 5.41e-01 & 6.56e-01 & 1.94e+01 & 1.48e-17 & 7.45e-17 & 1.72e-05 & 8.84e-17 & 8.57e-17 & 1.55e-04 \cr\hline
5e+03 & 1e+04 & 1.0e-01 &9.97e+01 & 2.08e-01 & 1.39e+01 & 4.26e-01 & 2.65e+00 & 5.54e+01 & 5.55e-17 & 4.16e-16 & 5.02e-07 & 1.67e-14 & 3.20e-16 & 2.29e-04 \cr\hline
5e+03 & 1e+05 & 1.0e-01 &7.64e+01 & 7.24e-01 & 1.42e+02 & 1.12e+00 & 8.51e+00 & 7.85e+01 & 2.36e-17 & 9.31e-11 & 6.38e-05 & 3.13e-16 & 1.79e-16 & 1.58e-04 \cr\hline
5e+03 & 5e+05 & 1.0e-01 &2.30e+02 & 6.97e+00 & 6.54e+02 & 7.02e+00 & 1.52e+01 & 1.70e+02 & 1.52e-17 & 1.87e-10 & 3.73e-05 & 3.92e-16 & 1.68e-16 & 1.48e-04 \cr\hline
2e+04 & 1e+05 & 1.0e-02 &6.32e-01 & 9.46e-01 & 5.68e+00 & 1.05e+00 & 2.49e+00 & 4.28e+01 & 1.36e-17 & 3.55e-06 & 4.33e-07 & 1.99e-06 & 1.28e-16 & 1.42e-04 \cr\hline
2e+04 & 5e+05 & 1.0e-02 &6.66e-01 & 4.46e+00 & 3.78e+01 & 5.63e+00 & 9.28e+00 & 1.23e+02 & 8.48e-18 & 3.37e-06 & 8.83e-07 & 1.36e-06 & 2.89e-16 & 1.10e-04 \cr\hline
2e+04 & 1e+06 & 1.0e-02 &1.85e+00 & 9.30e+00 & 6.50e+01 & 1.17e+01 & 1.59e+01 & 2.06e+02 & 7.08e-18 & 4.34e-06 & 6.27e-06 & 1.76e-06 & 9.65e-17 & 1.12e-04 \cr\hline
1e+05 & 1e+07 & 1.0e-03 &7.38e+00 & 1.06e+01 & 6.14e+00 & 9.35e+01 & 9.60e+01 & 1.56e+03 & 1.39e-18 & 1.39e-18 & 1.39e-18 & 1.76e-17 & 1.76e-17 & 5.90e-05 \cr\hline
\end{tabular}

}
\caption{LP application results averaged on $5$ randomly generated problems per row.}
\label{table:LPAppl}
}
\end{table}

\begin{figure}[H]
\centering
\includegraphics[width = 0.45\linewidth, height =
	0.25\textheight, keepaspectratio]{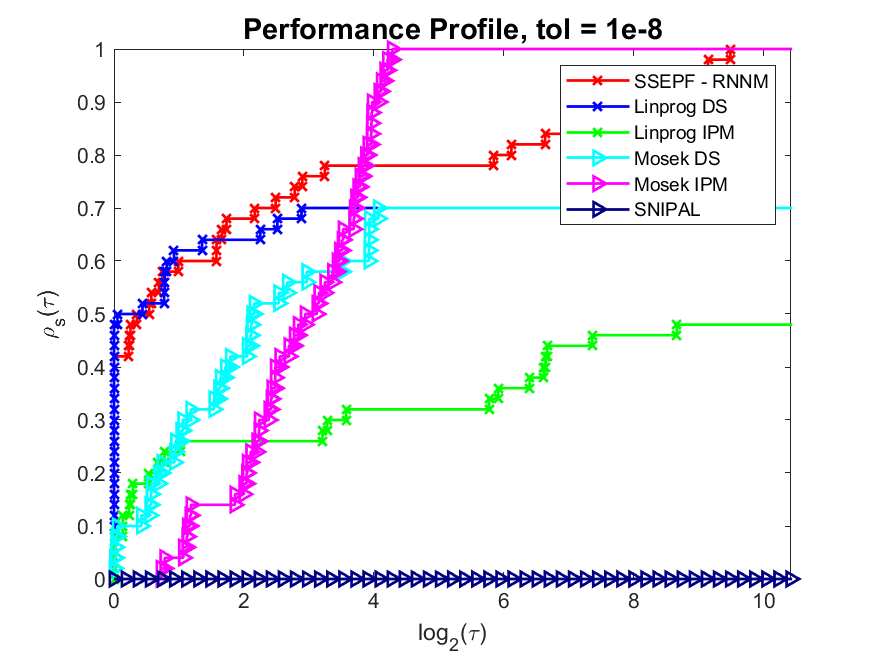}
\caption{Performance Profiles for LP application with respect to all problems.}
\label{fig: LPperf}
\end{figure}

{\crb
We also consider the first five problems in alphabetical order
from the subset of the NETLIB dataset where primal strict feasibility (PSF)
holds~\cite[Sect. 4.2.2]{ImWolk:22}.
We then check dual strict feasibility (DSF) and 
include the value of the constant we obtain
from solving the theorem of the alternative,
i.e.,~a large, respectively small, constant indicates an algebraically
\emph{fat}, respectively \emph{thin}, feasible set.
Failure, or near failure, of strict feasibility correlates with the
difficulty of the numerics. We successfully solve two of the five
problems. We think that the difficulties from the NETLIB dataset is due
to the dual feasible set being very thin for some problems. For example,
in \Cref{table:LPNetlibStrictFeas}, the problems 25fv47 and lotfi have a
very thin feasible set in the dual problem.

{\crb It is important to note that the performance of SSEPF-RNNM on the
blend problem is signifcantly worse than the other solvers. A common
issue with SSEPF-RNNM when solving the blend problem as well as rows
$4$-$6$ of \Cref{table:LPAppl} is that at certain tolerances, \RNNM uses
the maximum number of iterations ($2000$) to solve the \BAP subproblem.
In other words, even though we are performing a warm-start with the
solution from the previous \BAP subproblem, \RNNM can fail to converge
to the desired relative tolerance. However, even though \RNNM failed to
converge, it still provides a solution that is very close to the optimal
solution, i.e., instead of solving the \BAP subproblem to within a
relative tolerance of $10^{-14}$, it returns a solution that is within a
relative tolerance of $10^{-12}$ or $10^{-13}$. 
There are at least two solutions to this issue.
First, we can decrease the length of the Newton step when the iteration
count is large. Using this heuristic shows significant improvement in
performance when solving the blend problem.
Secondly, if \RNNM fails to converge to within the specified relative
tolerance of $10^{-14}$, we can try a larger relative tolerance, 
such as $10^{-13}$. This strategy has shown to be crucial when trying to solve
problems like 25fv47, where we are not able to solve the \BAP subproblem
with high accuracy due to it's thin dual feasible set.}
\begin{table}[H]
{\small
\flushleft
\scalebox{0.90}{
\begin{tabular}{|c||c|c|} \hline
\multicolumn{1}{|c||}{Problem:} & \multicolumn{1}{|c|}{Primal Strict Feas.} & \multicolumn{1}{|c|}{Dual Strict Feas.} \cr\hline
         25fv47 &2.00e-01 & 2.01e-17 \cr\hline
          afiro &9.00e+00 & 1.19e-01 \cr\hline
          blend &7.30e-02 & 3.49e-03 \cr\hline
         israel &3.71e+00 & 1.38e-03 \cr\hline
          lotfi &1.00e+00 & 1.89e-10 \cr\hline
\end{tabular}

}
\caption{Primal and Dual strict feasibility of NETLIB problems.}
\label{table:LPNetlibStrictFeas}
}
\end{table}

\begin{table}[H]
{\small
\flushleft
\scalebox{0.60}{
\begin{tabular}{|c||cccccc|cccccc|} \hline
\multicolumn{1}{|c||}{} & \multicolumn{6}{|c|}{Time (s)} & \multicolumn{6}{|c|}{Rel. Resids.}\cr\cline{1-13}
        Problem: &  SSEPF-RNNM & Linprog DS & Linprog IPM & MOSEK DS & MOSEK IPM &  SNIPAL &    SSEPF-RNNM & Linprog DS & Linprog IPM & MOSEK DS & MOSEK IPM &  SNIPAL \cr\hline
         25fv47  &  Inf & 2.01e-01 & 1.01e-01 & 3.76e-01 & 1.54e-01 & 1.20e+01 &     Inf & 2.30e-15 & 2.25e-15 & 5.51e-16 & 1.09e-14 & 7.36e-05 \cr\hline
          afiro & 2.62e-02 & 7.71e-03 & 2.91e-03 & 9.16e-02 & 9.01e-02 & 9.81e-02 & 1.97e-16 & 3.67e-16 & 8.62e-14 & 7.49e-17 & 1.43e-13 & 9.39e-11 \cr\hline
          blend & 1.42e+02 & 8.48e-03 & 3.81e-03 & 9.12e-02 & 9.03e-02 & 1.58e+00 & 5.37e-15 & 4.78e-14 & 1.31e-13 & 1.33e-15 & 1.63e-15 & 1.30e-03 \cr\hline
         israel &  Inf & 1.07e-02 & 2.79e-02 & 9.33e-02 & 9.82e-02 & 3.27e+00 &     Inf & 7.15e-16 & 8.44e-14 & 6.57e-16 & 8.93e-12 & 5.21e-05 \cr\hline
          lotfi &   Inf & 9.63e-03 & 7.86e-03 & 9.41e-02 & 9.43e-02 & 2.00e+00 &     Inf & 4.61e-14 & 3.38e-14 & 1.17e-16 & 9.05e-13 & 4.35e-05 \cr\hline
\end{tabular}

}
\caption{LP application results on the NETLIB problems.}
\label{table:LPNetlib}
}
\end{table}

\begin{figure}[H]
\centering
\includegraphics[width = 0.45\linewidth, height =
	0.25\textheight, keepaspectratio]{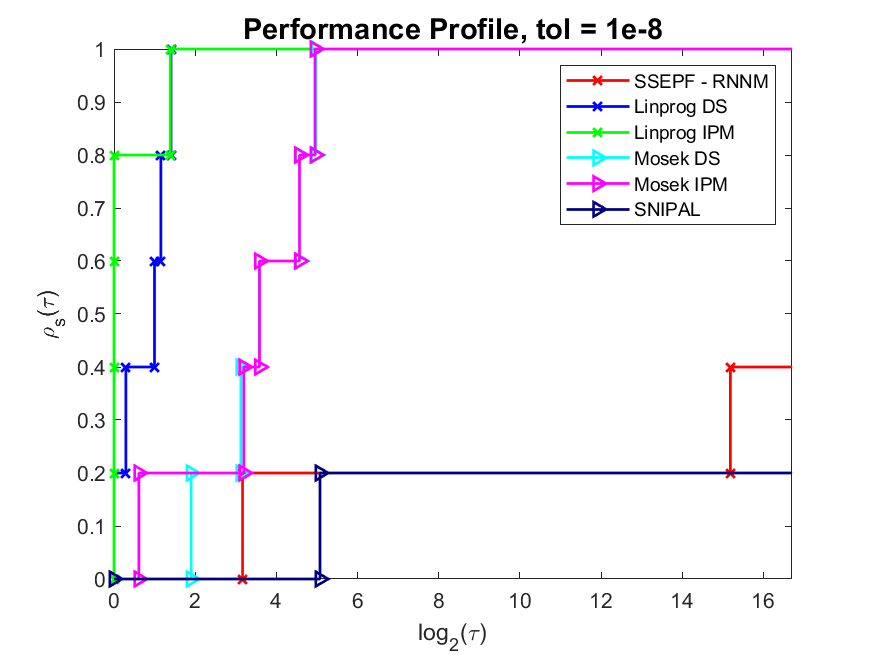}
\caption{Performance Profiles for LP application with respect to the Netlib problems.}
\label{fig: LPNetlibperf}
\end{figure}

Our algorithm has difficulties with highly
degenerate problems where the optimal solution
is not unique. Moreover, the optimal solution of minimum norm that our
algorithm finds can fail strict complementarity with many $x_i+z_i=0$.
The loss of strict complementarity results in a generalized Jacobian
with low rank as few columns of $A$ are chosen in~\cref{eq:maxrankM}.
Additionally, the sensitivity analysis of~\Cref{thm:pathfollowR} has difficulty
increasing $R$. Finally, the failure of strict complementarity indicates that
the gradient at optimality is not in the relative interior of the normal
cone, \Cref{lem:xoptvertex}, \Cref{item:nonnegpolar}, indicating failure
of differentiability of the projection.
}

\section{Conclusion}
\label{sect:concl}
In this paper we considered the theory and applications of the ``best
approximation problem'' of finding the projection of a point
onto a polyhedral set. We studied an elegant optimality condition,
derived using the Moreau decomposition, that
allowed for a, possibly both nonsmooth and singular,
Newton type method. However, this needed a 
perturbation of a max-rank choice of a generalized Jacobian, i.e.,~application
of nonsmooth analysis and regularization. The regularization guaranteed
a descent direction but the method was not necessarily monotonically
decreasing. We presented extensive comparisons with the \HLWB algorithm 
approach, e.g.,~\cite{1996_Ba}, and found that, in our experiments,
our method outperformed \HLWB in both speed and accuracy.

We discussed several applications including solving large, sparse,
linear programs. The preliminary tests we performed
were very efficient and outperformed the other codes 
we used for comparison both in speed
and accuracy. Our algorithmic approach can be
considered as a \textdef{stepping stone external path following} method
since we
follow an external path with parameter $R$ in the objective function; but
we only consider a discrete number of points on the path found using
sensitivity analysis. We discovered that very few stepping stones are needed,
often just one suffices.

\vspace{.5in}
 \label{page:acknow}
{\crb \textbf{Acknowledgements.} We thank the referees for carefully reading
the paper and for their helpful comments.}

\appendix

\section{Pseudocodes for Generalized Simplex}
\label{app:pseudos}

The pseudocodes described in 
\Cref{alg:ExactNonsmoothNewton,alg:InexactNonsmoothNewton,alg:extHLWB}
solve
\Cref{eq:projproblemhyp} using the exact and inexact nonsmooth Newton
methods \RNNMp, respectively. 

\begin{algorithm}
	\caption{\BAP of $v$ for constraints $Ax=b,x\geq 0$; exact 
	Newton direction}
	\label{alg:ExactNonsmoothNewton}
	\begin{algorithmic}[1]
		\REQUIRE $v \in \Rn,y_0\in \Rm, \, (A \in \R^{m \times
		n},\rank(A)=m),\, b \in \R^m, \, \varepsilon>0$, maxiter $\in \N$.
		\STATE \textbf{Output.} Primal-dual opt.: 
                $x_{k+1},(y_{k+1},z_{k+1})$
		\STATE \textbf{Initialization.} 
		$k \gets 0$, 
		$x_0 \gets (v+A^Ty_0)_+$,
		$z_0 \gets (x_0-(v+A^Ty_0))_+$,\\
		\qquad \qquad\qquad  $F_0 = Ax_0-b$,
		stopcrit $\gets \norm{F_0} / (1 + \norm{b})$ 
		\WHILE{((stopcrit $> \varepsilon) \,\&\, (k\leq
                $ maxiter)) }
				\STATE $V_k = \sum_{i \in\cI_+} A_i A_i^T
	{\crb	+ \sum_{i \in \bar \cI_0} \frac 1{\|A_i\|^2}A_i A_i^T}$ \label{line:StartCompA1}
			\STATE $\lambda = \min(1e^{-3}, \text{ stopcrit})$  
				\STATE $\bar V = (V_k+\lambda I_m)$ 
			\STATE $\text{solve pos. def. system } \bar V d = -F_k$ for Newton direction $d$  \label{line:solvepd}
			\STATE \textbf{updates}  
			\begin{ALC@g}
			\STATE $ y_{k+1} \gets y_k + d$  
			\STATE $x_{k+1} \gets (v + A^Ty_{k+1})_+$ 
			\STATE $z_{k+1} \gets  (x_{k+1}-(v+A^Ty_{k}))_+$ 
			\STATE $F_{k+1} \gets Ax_{k+1} - b$ \text{  (residual)}  \label{line:EndCompA1}
			\STATE stopcrit $\gets \norm{F_{k+1}} / (1+\norm{b})$  
\label{line:stopcrit}
			\STATE $k \gets k + 1$ 
			\end{ALC@g}
		\ENDWHILE
	\end{algorithmic}		
\end{algorithm}

\begin{algorithm}
	\caption{\BAP of $v$ for constraints $Ax=b,x\geq 0$, inexact Newton direction}
	\label{alg:InexactNonsmoothNewton}
	\begin{algorithmic}[1]
		\REQUIRE $v \in \Rn,y_0\in \Rm, \, (A \in \R^{m \times
		n},\rank(A)=m),\, b \in \R^m, \, \varepsilon>0$, maxiter $\in \N$.
		\STATE \textbf{Output.} Primal-dual: $x_{k+1},(y_{k+1},z_{k+1})$
		\STATE \textbf{Initialization.} 
		$k \gets 0$, 
		$x_0 \gets (v+A^Ty_0)_+$,
		$z_0 \gets (x_0-(v+A^Ty_0))_+$,\\
		\qquad \qquad\qquad $\delta \in (0,1], \, \nu \in [1 + \frac{\delta}2,2]$, and a sequence $\theta$ such that $\theta_k \geq 0$ and $\sup_{k \in \mathbb{N}} \theta_k < 1$\ \\ 
		\qquad \qquad\qquad  $F_0 = Ax_0-b$,
		stopcrit $\gets \norm{F_0} / (1 + \norm{b})$ \\
		\WHILE{((stopcrit $> \varepsilon) \,\&\, (k\leq
                $ maxiter))} 
			\STATE $V_k = \sum_{i \in\cI_+} A_iA_i^T
	{\crb	+ \sum_{i \in \bar \cI_0} \frac 1{\|A_i\|^2}A_i A_i^T}$
			\STATE $\lambda = (\text{stopcrit})^\delta$  
			\STATE $\bar V = (V_k+\lambda I_m)$ 
			\STATE $\text{solve } \bar V d = -F_k$ for 
  Newton direction $d$ such that residual
			$\norm{r_k} \leq \theta_k \norm{F_k}^\nu$ 
			\STATE \textbf{updates}  
			\begin{ALC@g}
			\STATE $ y_{k+1} \gets y_k + d$  
			\STATE $x_{k+1} \gets (v + A^Ty_{k+1})_+$ 
			\STATE $z_{k+1} \gets  (x_{k+1}-(v+A^Ty_{k}))_+$ 
			\STATE $F_{k+1} \gets Ax_{k+1} - b$ \text{  (residual)} 
			\STATE stopcrit $\gets \norm{F_{k+1}} / (1+\norm{b})$  
			\STATE $k \gets k + 1$
			\end{ALC@g}
		\ENDWHILE
	\end{algorithmic}		
\end{algorithm}
		
\begin{algorithm} 
	\caption{Extended \HLWB algorithm}
	\label{alg:extHLWB}
	\begin{algorithmic}[1]
		\REQUIRE $v \in \Rn, (A \in \R^{m \times
		n},\rank(A)=m),\, b \in \R^m, \, \varepsilon>0$, maxiter $\in \N$.  
		\STATE \textbf{Output.}  $x_{k+1}$
		\STATE \textbf{Initialization.} 
		$k \gets 0$, $msweeps \gets 0$ 
               	$x_0 \gets max(v,0)$, $\hat x_0 \gets x_0$, $i_0 = 1$ \\ 
		\qquad\qquad  \qquad 
		stopcrit $\gets \norm{A\hat x_0 - b} / (1 + \norm{b})$ ($=\norm{F_0} / (1 + \norm{b})$)
		\WHILE{((stopcrit $> \varepsilon) \,\&\, (k\leq
                $ maxiter))}  
			\IF{$1 \leq i_k \leq m$} 
				\STATE $\hat x_k = x_k +
				\frac {b_{i_k}-
				a_{i_k}^Tx_k}{\|a_{i_k}\|^2}a_{i_k}$ \label{line:StartCompA3}
			\ELSE
				\STATE $\hat x_k =\max(0,x_k)$ 
			\ENDIF
			\STATE \textbf{updates}  
			\begin{ALC@g}
			\STATE  $\sigma_k = \frac{1}{k+1}$ 
			\STATE  $ x_{k+1} \gets \sigma_k v + (1 -
			\sigma_k) \hat x_k$ 
			\STATE  stopcrit $\gets  \norm{A\hat x_k - b} / (1 + \norm{b})$ \label{line:EndCompA3}

			 \IF{$k (\text{mod } m + 1) = 0$} \label{line:A3startIf}
				\STATE $msweeps = msweeps + 1$ \label{line:A3middleIf}
			 \ENDIF \label{line:A3endIf}
			\STATE  $ i_k = k (\text{mod } m) + 1$
			 \end{ALC@g}
		\ENDWHILE
	\end{algorithmic}
\end{algorithm}

\section{Additional Performance Profiles}

\subsection{Nondegenerate}
\label{subsect: nondeg perf}
\begin{figure}[H]
\centering

\begin{subfigure}[t]{0.95\linewidth}
	\centering
	\includegraphics[height =
	0.26\textheight, keepaspectratio]{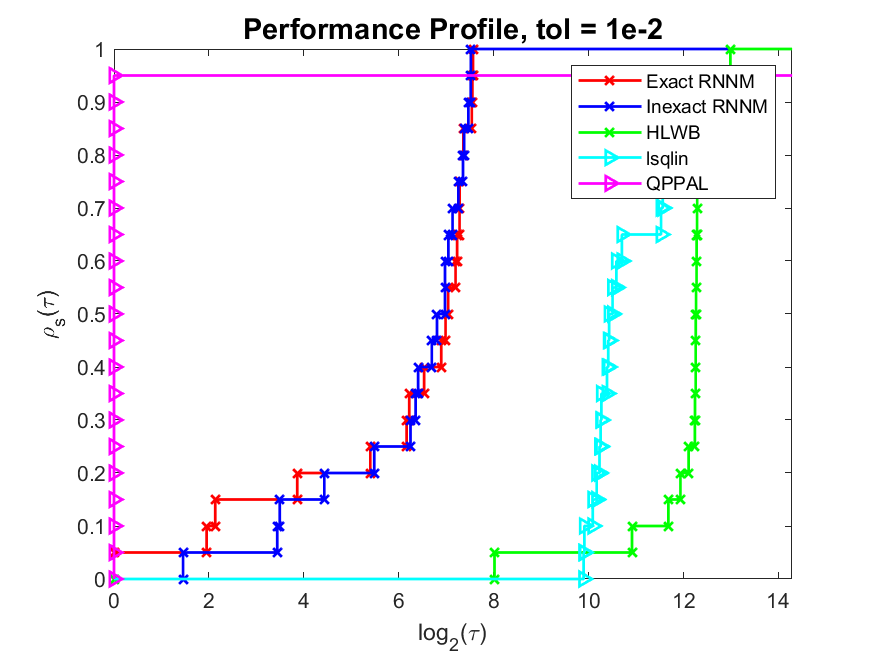}
	\label{figapp: nondegenerate PFsizeM1}
	\caption{tol = $10^{-2}$}
\end{subfigure}
\begin{subfigure}[t]{0.95\linewidth}
	\centering
	\includegraphics[height =
	0.26\textheight, keepaspectratio]{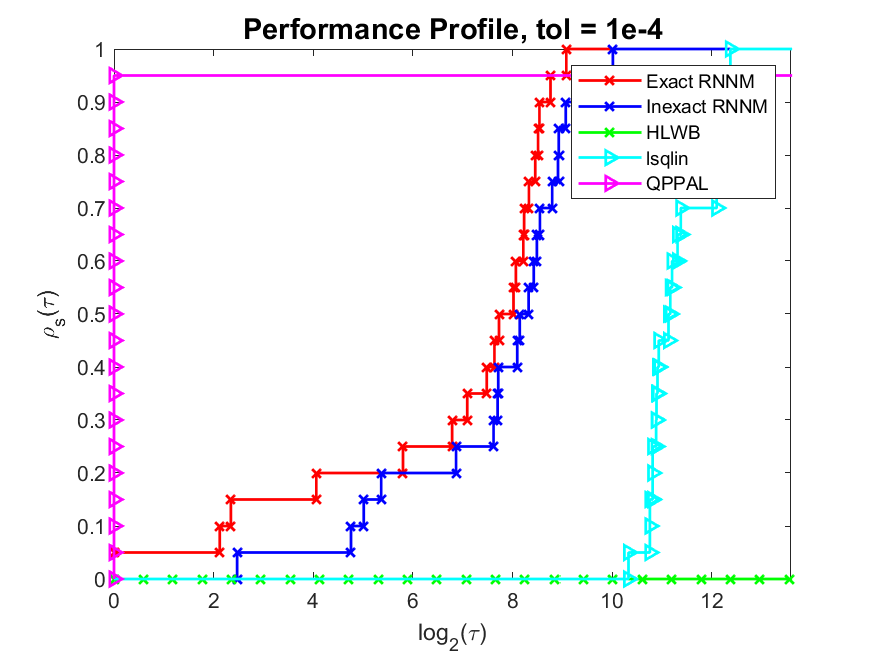}
	\label{figapp: nondegenerate PFsizeM2}
	\caption{tol = $10^{-4}$}
\end{subfigure}

\begin{subfigure}[t]{0.95\linewidth}
	\centering
	\includegraphics[height =
	0.26\textheight, keepaspectratio]{PFsizeM3.png}
	\label{figapp: nondegenerate PFsizeM3}
	\caption{tol = $10^{-14}$}
\end{subfigure}

\caption{Performance Profiles for varying $m$ for nondegenerate vertex solutions.}
\end{figure}

\newpage

\begin{figure}[H]
\centering

\begin{subfigure}[t]{0.95\linewidth}
	\centering
	\includegraphics[height =
	0.26\textheight, keepaspectratio]{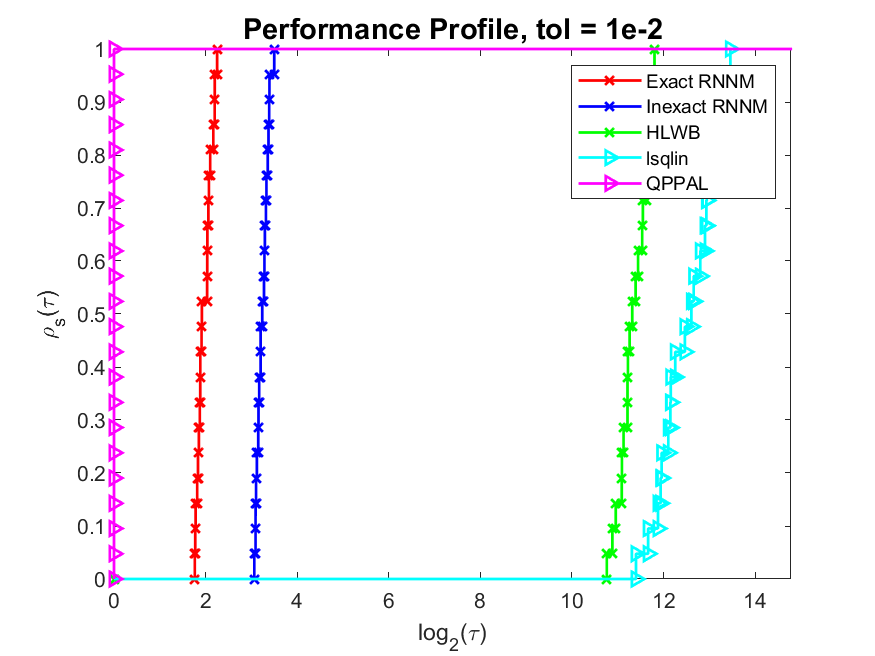}
	\label{figapp: nondegenerate PFsizeN1}
	\caption{tol = $10^{-2}$}
\end{subfigure}
\begin{subfigure}[t]{0.95\linewidth}
	\centering
	\includegraphics[height =
	0.26\textheight, keepaspectratio]{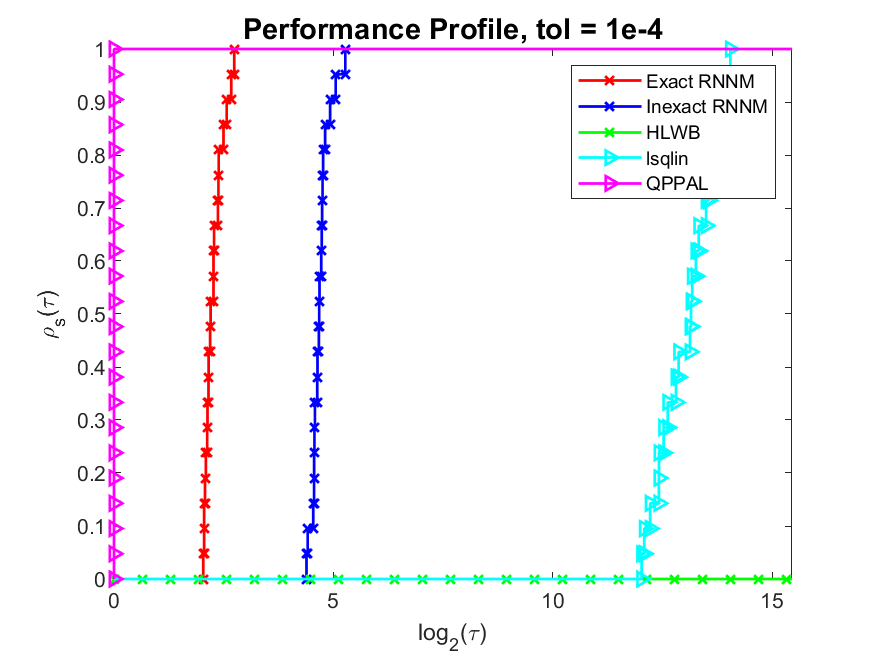}
	\label{figapp: nondegenerate PFsizeN2}
	\caption{tol = $10^{-4}$}
\end{subfigure}

\begin{subfigure}[t]{0.95\linewidth}
	\centering
	\includegraphics[height =
	0.26\textheight, keepaspectratio]{PFsizeN3.png}
	\label{figapp: nondegenerate PFsizeN3}
	\caption{tol = $10^{-14}$}
\end{subfigure}

\caption{Performance Profiles for varying $n$ for nondegenerate vertex solutions.}
\end{figure}

\newpage

\begin{figure}[H]
\centering

\begin{subfigure}[t]{0.95\linewidth}
	\centering
	\includegraphics[height =
	0.26\textheight, keepaspectratio]{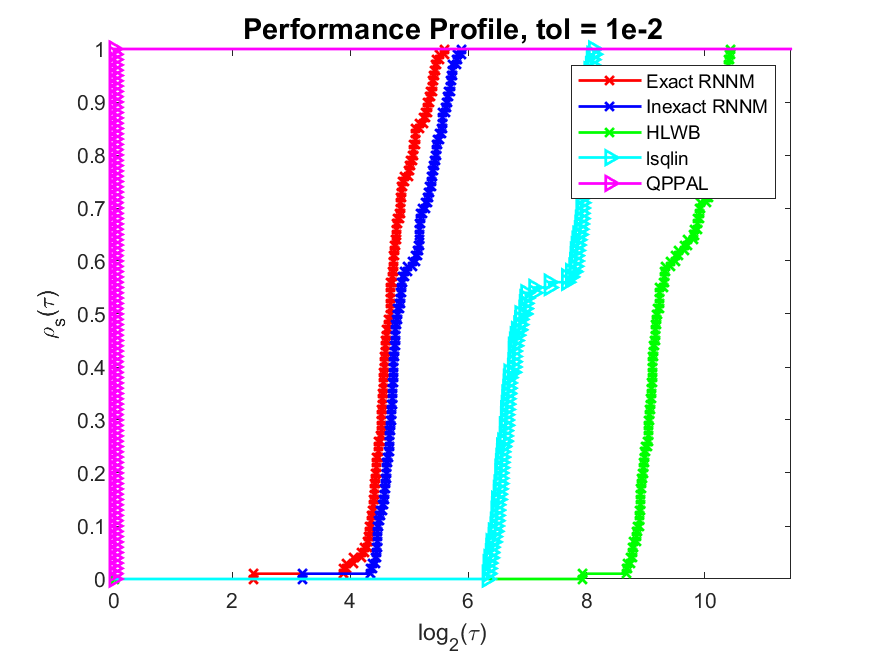}
	\label{figapp: nondegenerate PFdense1}
	\caption{tol = $10^{-2}$}
\end{subfigure}
\begin{subfigure}[t]{0.95\linewidth}
	\centering
	\includegraphics[height =
	0.26\textheight, keepaspectratio]{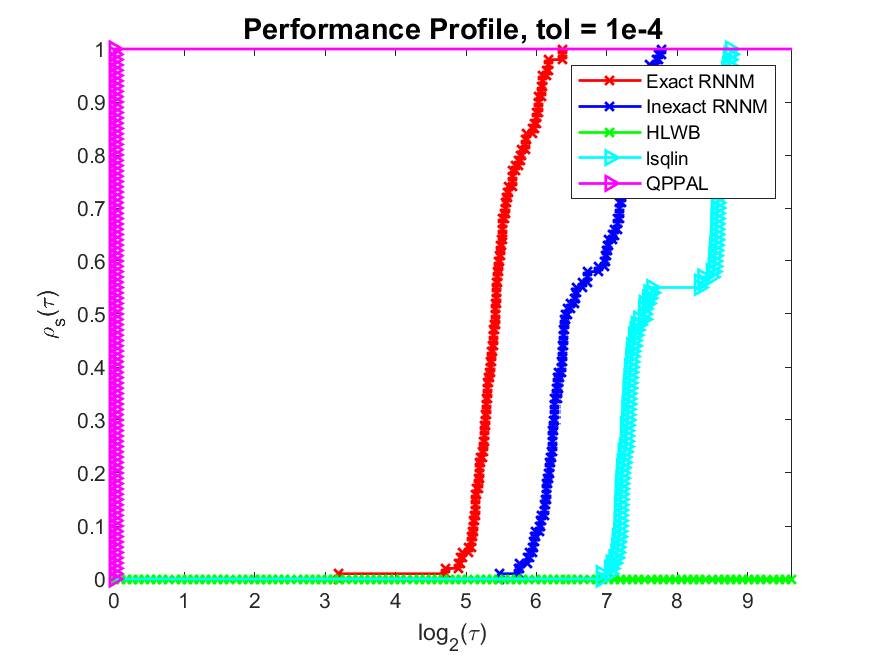}
	\label{figapp: nondegenerate PFdense2}
	\caption{tol = $10^{-4}$}
\end{subfigure}

\begin{subfigure}[t]{0.95\linewidth}
	\centering
	\includegraphics[height =
	0.26\textheight, keepaspectratio]{PFdense3.png}
	\label{figapp: nondegenerate PFdense3}
	\caption{tol = $10^{-14}$}
\end{subfigure}

\caption{Performance Profiles for varying density for nondegenerate vertex solutions.}
\end{figure}

\subsection{Degenerate}

\begin{table}[H]
{\small
\flushleft
\caption{Varying problem sizes $m$ and comparing computation time with relative residual for degenerate vertex solutions.}
\scalebox{0.8}{
\begin{tabular}{|ccc||ccccc|ccccc|} \hline
\multicolumn{3}{|c||}{Specifications} & \multicolumn{5}{|c|}{Time (s)} & \multicolumn{5}{|c|}{Rel. Resids.}\cr\cline{1-13}
  $m$&   $n$& \% density&  Exact & Inexact &    HLWB &  lsqlin &   QPPAL &   Exact & Inexact &    HLWB &  lsqlin &   QPPAL \cr\hline
  500 &  3000 & 8.1e-01 &4.23e-02 & 1.51e-01 & 1.54e+02 & 3.77e+00 & 1.14e+00 &  1.96e-16 & 8.26e-16 & 2.25e-04 & 7.26e-17 & 1.72e-17 \cr\hline
 1000 &  3000 & 8.1e-01 &4.40e-01 & 9.97e-01 & 3.71e+02 & 5.37e+00 & 2.15e+00 &  2.70e-16 & 1.95e-15 & 2.14e-04 & 3.87e-17 & 2.70e-17 \cr\hline
 1500 &  3000 & 8.1e-01 &1.17e+00 & 3.23e+00 & 6.09e+02 & 7.02e+00 & 4.69e+00 &  3.41e-17 & 6.73e-16 & 2.27e-04 & 3.95e-17 & 1.16e-17 \cr\hline
 2000 &  3000 & 8.1e-01 &2.49e+00 & 7.51e+00 & 8.67e+02 & 1.02e+01 & 7.81e+00 &  6.11e-17 & 3.11e-17 & 2.24e-04 & 3.14e-17 & -2.74e-17 \cr\hline
\end{tabular}

}
\label{table: degenerate sizeM3}
}
\end{table}

\begin{table}[H]
{\small
\flushleft
\caption{Varying problem sizes $n$ and comparing computation time with relative residual for degenerate vertex solutions.}
\scalebox{0.8}{
\begin{tabular}{|ccc||ccccc|ccccc|} \hline
\multicolumn{3}{|c||}{Specifications} & \multicolumn{5}{|c|}{Time (s)} & \multicolumn{5}{|c|}{Rel. Resids.}\cr\cline{1-13}
  $m$&   $n$& \% density&  Exact & Inexact &    HLWB &  lsqlin &   QPPAL &   Exact & Inexact &    HLWB &  lsqlin &   QPPAL \cr\hline
  200 &  3000 & 8.1e-01 &3.12e-03 & 3.69e-02 & 4.45e+01 & 3.50e+00 & 8.66e-01 &  8.64e-18 & 7.39e-17 & 2.56e-04 & 6.52e-16 & 5.89e-17 \cr\hline
  200 &  3500 & 8.1e-01 &3.08e-03 & 4.05e-02 & 5.17e+01 & 4.93e+00 & 1.00e+00 &  9.07e-18 & 1.26e-17 & 2.78e-04 & 1.23e-15 & 2.15e-17 \cr\hline
  200 &  4000 & 8.1e-01 &3.24e-03 & 3.70e-02 & 5.82e+01 & 7.31e+00 & 1.09e+00 &  1.46e-16 & 8.91e-16 & 2.80e-04 & 3.21e-16 & -9.18e-18 \cr\hline
  200 &  4500 & 8.1e-01 &3.99e-03 & 4.17e-02 & 6.58e+01 & 1.01e+01 & 1.18e+00 &  1.80e-15 & 2.05e-16 & 3.13e-04 & 4.61e-17 & 1.71e-16 \cr\hline
\end{tabular}

}
\label{table: degenerate sizeN3}
}
\end{table}

\vspace{.1in}

\begin{table}[H]
{\small
\flushleft
\caption{Varying problem density and comparing computation time with relative residual for degenerate vertex solutions.}
\scalebox{0.8}{
\begin{tabular}{|ccc||ccccc|ccccc|} \hline
\multicolumn{3}{|c||}{Specifications} & \multicolumn{5}{|c|}{Time (s)} & \multicolumn{5}{|c|}{Rel. Resids.}\cr\cline{1-13}
  $m$&   $n$& \% density&  Exact & Inexact &    HLWB &  lsqlin &   QPPAL &   Exact & Inexact &    HLWB &  lsqlin &   QPPAL \cr\hline
  300 &  1000 &   25 &5.69e-02 & 2.66e-01 & 4.55e+01 & 3.30e-01 & 1.20e+00 &  2.83e-17 & 1.14e-17 & 1.50e-04 & 8.61e-17 & 5.99e-17 \cr\hline
  300 &  1000 &   50 &5.43e-02 & 2.28e-01 & 5.39e+01 & 3.08e-01 & 1.82e+00 &  1.23e-16 & 1.97e-17 & 1.44e-04 & 8.08e-16 & 1.42e-17 \cr\hline
  300 &  1000 &   75 &7.75e-02 & 2.86e-01 & 5.36e+01 & 3.16e-01 & 1.49e+01 &  4.83e-16 & 1.72e-17 & 1.62e-04 & 3.49e-16 & -3.43e-16 \cr\hline
  300 &  1000 &  100 &7.27e-02 & 2.47e-01 & 4.65e+01 & 3.00e-01 & 2.54e+02 &  5.66e-16 & 2.15e-17 & 1.63e-04 & 1.91e-15 & 1.04e-14 \cr\hline
\end{tabular}

}
\label{table: degenerate dense3}
}
\end{table}

\begin{figure}[H]
\centering

\begin{subfigure}[t]{0.95\linewidth}
	\centering
	\includegraphics[height =
	0.26\textheight, keepaspectratio]{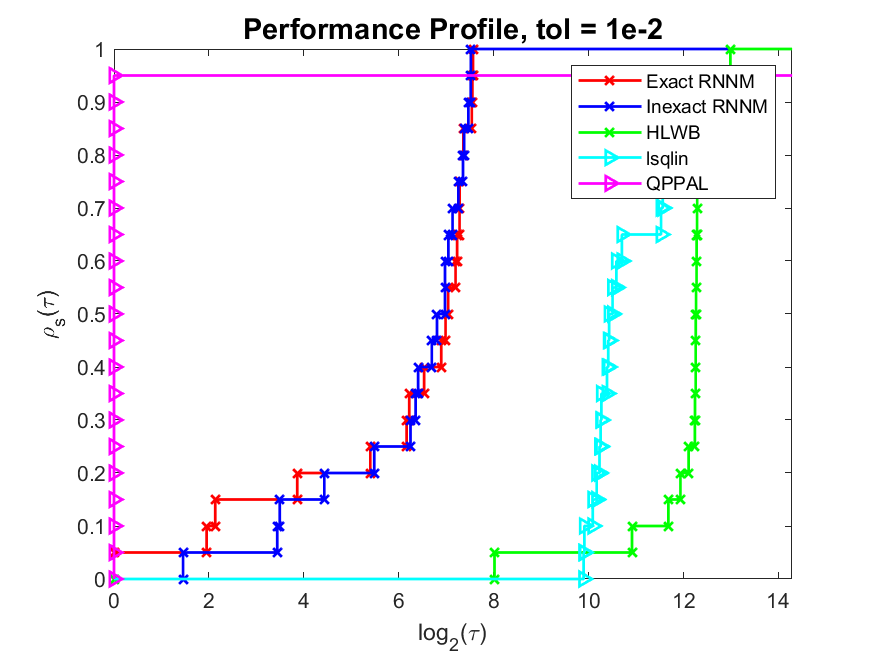}
	\label{figapp: degenerate PFsizeM1}
	\caption{tol = $10^{-2}$}
\end{subfigure}
\begin{subfigure}[t]{0.95\linewidth}
	\centering
	\includegraphics[height =
	0.26\textheight, keepaspectratio]{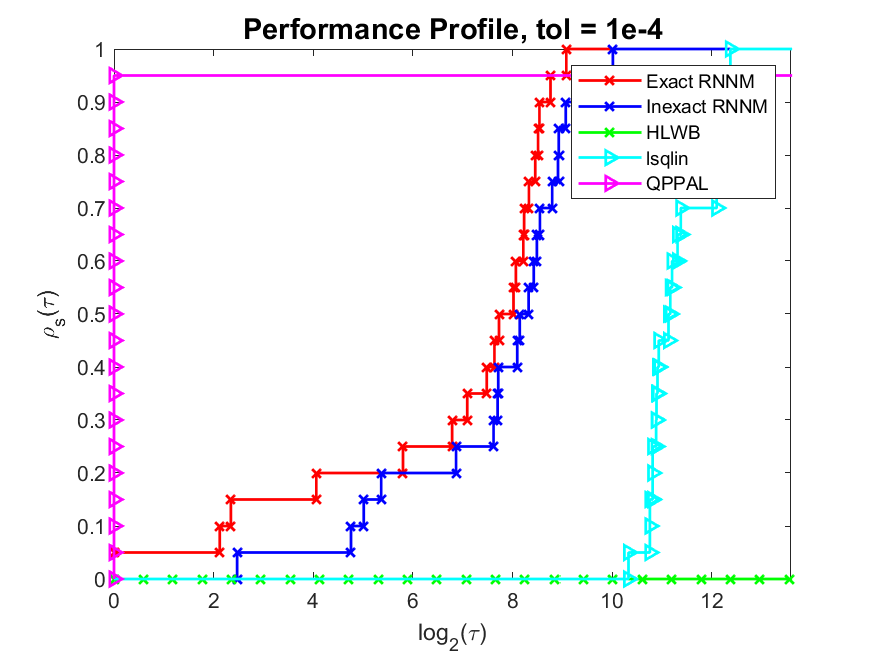}
	\label{figapp: degenerate PFsizeM2}
	\caption{tol = $10^{-4}$}
\end{subfigure}

\begin{subfigure}[t]{0.95\linewidth}
	\centering
	\includegraphics[height =
	0.26\textheight, keepaspectratio]{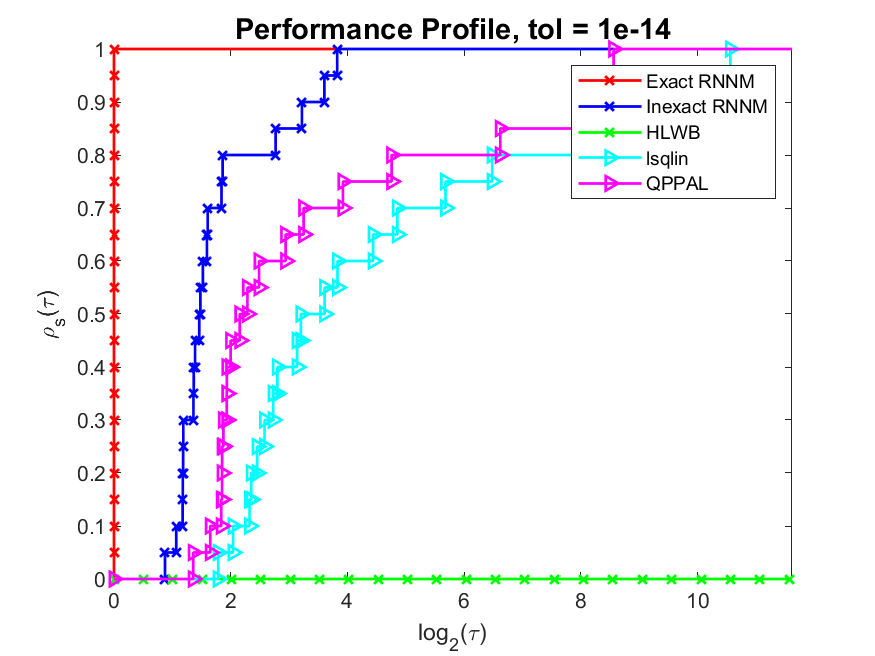}
	\label{figapp: degenerate PFsizeM3}
	\caption{tol = $10^{-14}$}
\end{subfigure}

\caption{Performance Profiles for varying $m$ for degenerate vertex solutions.}
\end{figure}

\newpage

\begin{figure}[H]
\centering

\begin{subfigure}[t]{0.95\linewidth}
	\centering
	\includegraphics[height =
	0.26\textheight, keepaspectratio]{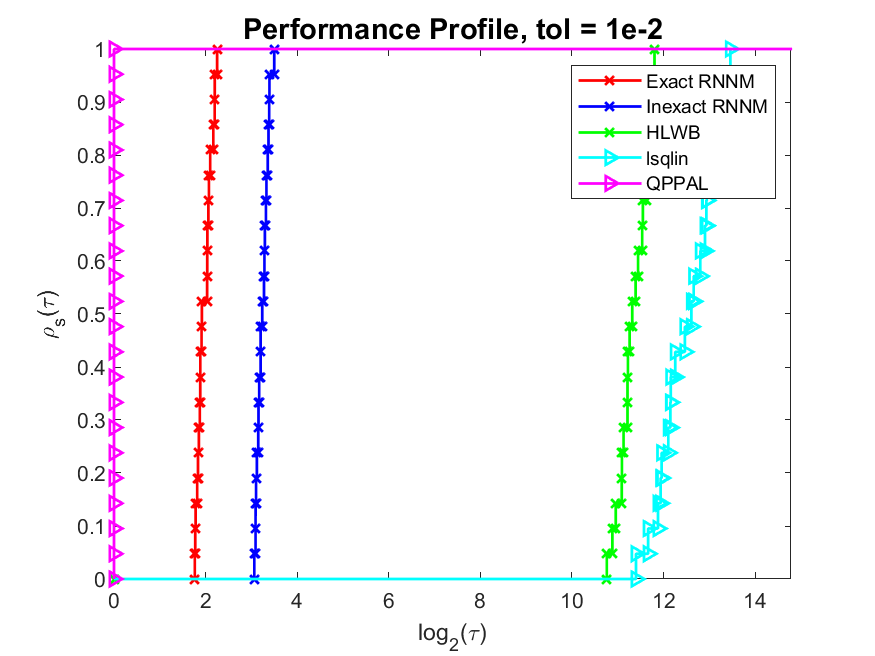}
	\label{figapp: degenerate PFsizeN1}
	\caption{tol = $10^{-2}$}
\end{subfigure}
\begin{subfigure}[t]{0.95\linewidth}
	\centering
	\includegraphics[height =
	0.26\textheight, keepaspectratio]{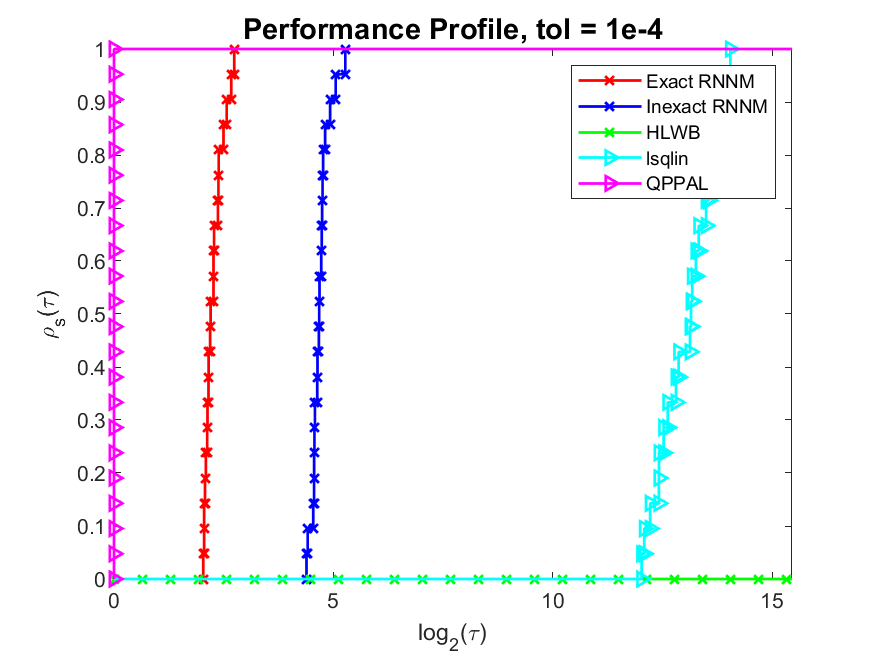}
	\label{figapp: degenerate PFsizeN2}
	\caption{tol = $10^{-4}$}
\end{subfigure}

\begin{subfigure}[t]{0.95\linewidth}
	\centering
	\includegraphics[height =
	0.26\textheight, keepaspectratio]{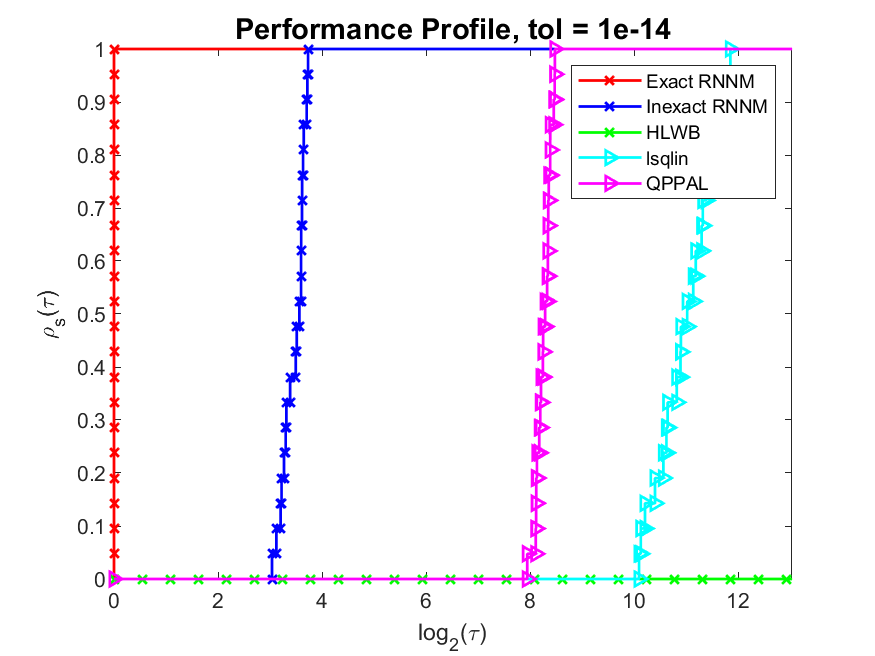}
	\label{figapp: degenerate PFsizeN3}
	\caption{tol = $10^{-14}$}
\end{subfigure}

\caption{Performance Profiles for varying $n$ for degenerate vertex solutions.}
\end{figure}

\newpage

\begin{figure}[H]
\centering

\begin{subfigure}[t]{0.95\linewidth}
	\centering
	\includegraphics[height =
	0.26\textheight, keepaspectratio]{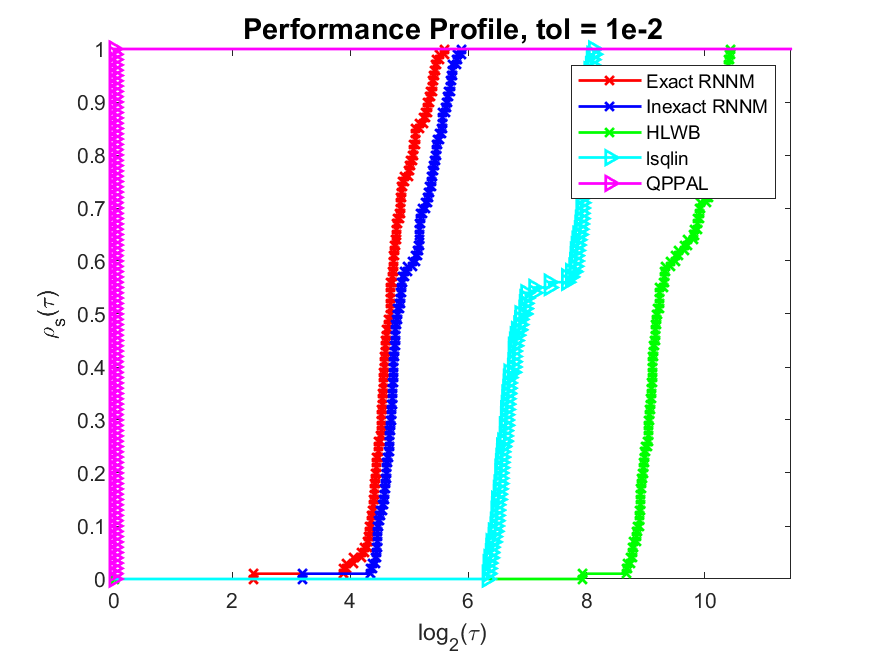}
	\label{figapp: degenerate PFdense1}
	\caption{tol = $10^{-2}$}
\end{subfigure}
\begin{subfigure}[t]{0.95\linewidth}
	\centering
	\includegraphics[height =
	0.26\textheight, keepaspectratio]{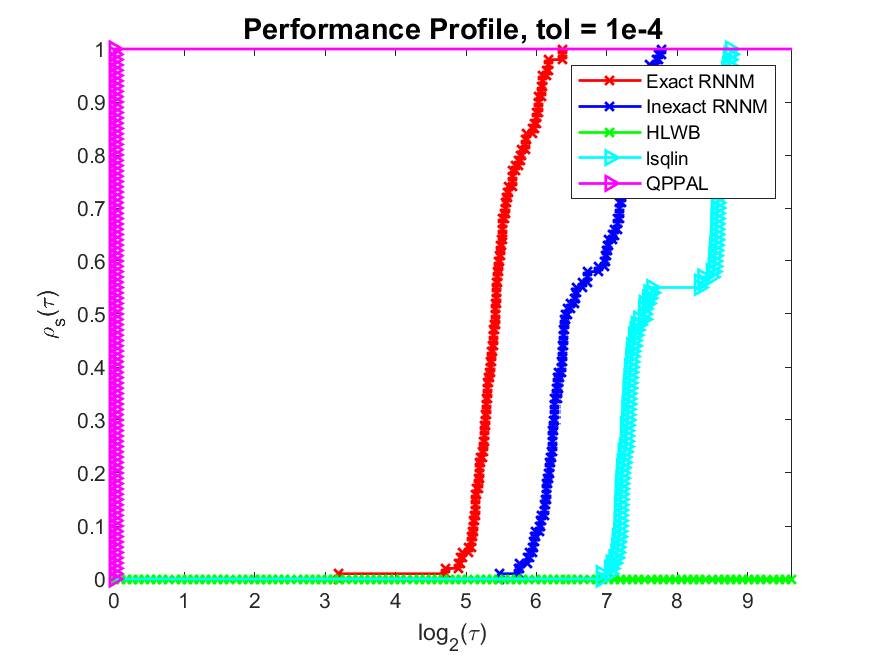}
	\label{figapp: degenerate PFdense2}
	\caption{tol = $10^{-4}$}
\end{subfigure}

\begin{subfigure}[t]{0.95\linewidth}
	\centering
	\includegraphics[height =
	0.26\textheight, keepaspectratio]{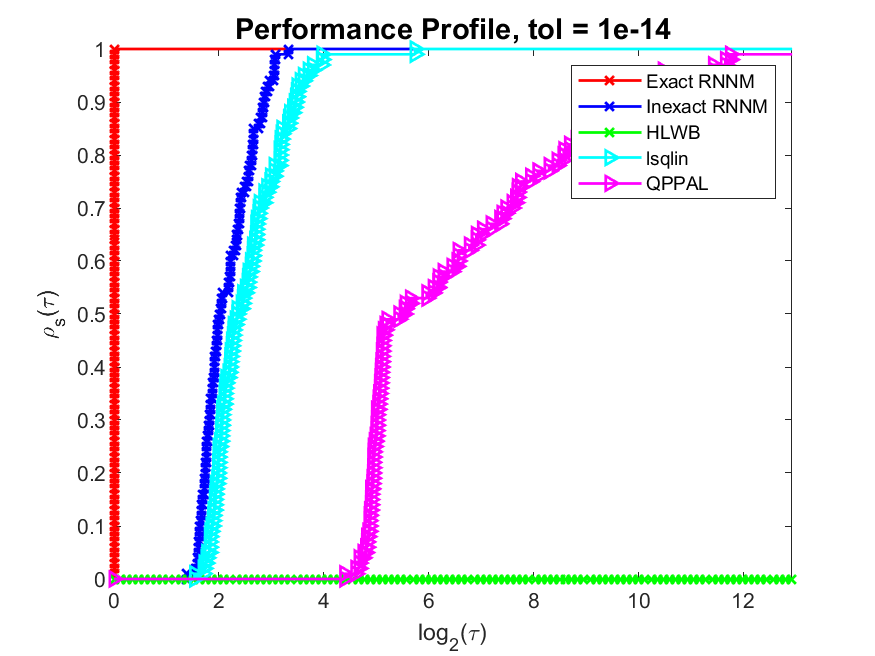}
	\label{figapp: degenerate PFdense3}
	\caption{tol = $10^{-14}$}
\end{subfigure}

\caption{Performance Profiles for varying density for degenerate vertex solutions.}
\end{figure}

\section{Applications of the \BAP and the \HLWB algorithm\label{subsec:Applications}}

The \BAP and the \HLWB algorithm play important roles in mathematical
and technological problems. We give two examples.

\begin{enumerate}
\item {\bf Finding best approximation pairs for two intersections of
closed convex sets} 
\\The problem of finding a best approximation pair
of two sets, which in turn generalizes the well-known convex feasibility
problem \cite{MR98f:90045}, has a long history that dates
back to work by Cheney and Goldstein in 1959 \cite{cheney-goldstein1959}.
This problem was recently revisited in \cite{ACJ2018} where an alternating
\HLWB (A-HLWB) algorithm was proposed and studied that can be used
when the two sets are finite intersections of half-spaces. Motivated
by that \cite{bauschke-on-ACJ-2022} presented alternative algorithms
that utilize projection and proximity operators. Their modeling framework
is able to accommodate even convex sets and their numerical experiments
indicate that these methods are competitive and in some cases superior
to the A-HLWB algorithm. The practical importance of the problem of
finding a best approximation pair of two sets stems from its relevance
to real-world situations wherein the feasibility-seeking modeling
is used and there are two disjoint constraints sets. One set represents
``hard'' constraints, i.e., constraints the must be met, while the
other set represents ``soft'' constraints which should be observed
as much as possible, see, e.g., \cite{hard-1999}. Under such circumstances,
the desire to find a point in the hard constraints set that will be
closest to the set of soft constraints leads to the problem of finding
a best approximation pair of the two sets.

\item
{\bf Least intensity modulated treatment plan in radiotherapy}
In the fully-discretized modelling of the intensity-modulated radiation
therapy (IMRT) treatment planning problem the irradiated body is
discretized into voxels and the external radiation field is discretized
into beamlets. This is represented by a system of linear inequalities
as in \cref{eq:Hihalfspace} with nonnegativity
constraints. The unknown vector $x$ represents radiation intensities
and if it is a solution of the linear feasibility problem then it
fulfills all the planning prescriptions dictated by the oncologist.
In such a feasibility-seeking approach several solutions are acceptable
but a solution that is closest to the origin will use the least possible
intensities that still fulfill the constraints. Delivering an acceptable
treatment plan with less radiation intensities is preferable and so
one replaces the feasibility-seeking problem by a \BAP of approximating
the origin by a point from the feasible sets, i.e., by seeking the
projection of the origin onto the feasible set. Such an approach was
used, e.g., in \cite{LIF-2003} where a simultaneous version of Hildreth’s
sequential algorithm for norm minimization over linear inequalities,
\cite{Hildreth,LC-hildreth1980}, \cite[Algorithm 6.5.2]{censor-zenios-book}
was combined with a norm-minimizing image reconstruction algorithm
of Herman and Lent \cite{ART4}, called ART4 (Algebraic Reconstruction
Technique 4), which handles in a special effective manner interval
inequalities.
\end{enumerate}

\section*{Data Availability and Conflict of Interest Statement}
The codes for generating both the data and the output is available at\\
\href{https://www.math.uwaterloo.ca/%7Ehwolkowi/henry/reports/ABSTRACTS.html}{
the paper link at URL www.math.uwaterloo.ca/\~{
}hwolkowi/henry/reports/ABSTRACTS.html} or by request from one of the
authors.

The authors declare no competing interests.

\bibliographystyle{plain}

\cleardoublepage
\label{ind:index}
\printindex
\addcontentsline{toc}{section}{Index}

\def\udot#1{\ifmmode\oalign{$#1$\crcr\hidewidth.\hidewidth
  }\else\oalign{#1\crcr\hidewidth.\hidewidth}\fi} \def\cprime{$'$}
  \def\cprime{$'$} \def\cprime{$'$} \def\cprime{$'$}

\addcontentsline{toc}{section}{Bibliography}
\label{endofpaperpg}

\end{document}